%% file: zeromodes_arxiv_2.tex
\documentclass[12pt,a4paper]{amsart}

 \usepackage{amssymb, amstext, amscd, amsmath, amsfonts, amsthm, amscd, color}

\usepackage{tikz,mathdots,enumerate}
\usepackage{pgfplots}
\usepackage{url}

\usepackage{tikz,mathdots,enumerate}

\usepackage{graphicx, array, blindtext}
\usepackage{url}
\usepackage{tikz}
\usetikzlibrary{shapes.geometric, calc}
\usepackage{caption}
\usepackage{array}
\usepackage{epsfig}
\usepackage{eucal}
\usepackage{latexsym}
\usepackage{mathrsfs}
\usepackage{textcomp}
\usepackage{verbatim}

\newtheorem{thm}{Theorem}[section]

\newtheorem{cor}[thm]{Corollary}

\newtheorem{lemma}[thm]{Lemma}

\newtheorem{prop}[thm]{Proposition}

\numberwithin{equation}{section}

\theoremstyle{definition}
\newtheorem{rem}[thm]{Remark}
\newtheorem{example}[thm]{Example}
\newtheorem{definition}[thm]{Definition}
\newcommand{\bC}{{\mathbb{C}}}

\newcommand{\bQ}{{\mathbb{Q}}}
\newcommand{\bR}{{\mathbb{R}}}
\newcommand{\bT}{{\mathbb{T}}}

\newcommand{\bZ}{{\mathbb{Z}}}

  \newcommand{\A}{{\mathcal{A}}}
  \newcommand{\B}{{\mathcal{B}}}
  \newcommand{\C}{{\mathcal{C}}}

  \newcommand{\F}{{\mathcal{F}}}
  \newcommand{\G}{{\mathcal{G}}}

  \newcommand{\J}{{\mathcal{J}}}

  \newcommand{\M}{{\mathcal{M}}}

\renewcommand{\P}{{\mathcal{P}}}

  \newcommand{\U}{{\mathcal{U}}}
  \newcommand{\V}{{\mathcal{V}}}

\newcommand{\ul}{\underline  }
\newcommand{\ol}{\overline  }

\usepackage{verbatim}

\usepackage{tikz}
\usetikzlibrary{calc}

\begin{document}


\title[Linear zero mode spectra for quasicrystals]{Linear zero mode spectra for quasicrystals}



\author[S. C. Power]{S. C. Power}

\address{Dept.\ Math.\ Stats.\\ Lancaster University\\
Lancaster LA1 4YF \\U.K. }

\email{s.power@lancaster.ac.uk}

\begin{abstract} A converse is given to the well-known
fact that a hyperplane localised zero mode of a crystallographic bar-joint framework gives rise to a line or lines in the zero mode (RUM) spectrum. These connections motivate definitions of  \emph{linear zero mode spectra}  for an aperiodic bar-joint framework $\G$ that are based on relatively dense sets of linearly localised flexes. For a Delone framework in the plane the  \emph{limit spectrum} ${\bf L}_{\rm lim}(\G,\ul{a})$ is defined in this way, as a subset of the reciprocal space for a reference basis $\ul{a}$ of the ambient space. A smaller spectrum,
the \emph{slippage spectrum} ${\bf L}_{\rm slip}(\G,\ul{a})$, 
is also defined.
For quasicrystal parallelogram frameworks associated with regular multi-grids, in the sense of de Bruijn and Beenker, these spectra coincide and are determined in terms of the geometry of $\G$. 
\end{abstract}


\maketitle


\section{Introduction}
A \emph{zero mode} of a crystallographic bond-node framework $\C$  in $\bR^d, d\geq 2,$ is considered here to be an excitation state of the nodes which has vanishing energy. These are also known as \emph{rigid unit modes} (\emph{RUMs}) or \emph{mechanical modes}. More precisely,  a zero mode is a simple harmonic motion oscillatory  state, with wave vector ${\bf k}$ in a reciprocal space $\bR^d_{\bf k}$, where the bonds are unstretched to first order.  Comparisons between simulations and experimental results have shown that the wave vectors of RUMs for simulated crystals coincide with those observed in material silicates and zeolites. See, for example, Giddy et al. \cite{gid-et-al}, 
Wegner \cite{weg} and the recent perspective of Dove \cite{dov-2019}. On the other hand, it is known that zero modes correspond to certain infinitesimal flexes \cite{pow-poly}. Also, the purely mathematical theory of infinitesimal flexibility and rigidity for periodic bar-joint frameworks is now well-developed as can be seen, for example, in the survey of Schulze and Whiteley \cite{sch-whi}. Moreover, recent articles
 have led to further mathematical insights into the spectra of zero modes. 
See, for example,
Badri, Kitson and Power \cite{bad-kit-pow-2}, Connelly, Shen and Smith \cite{con-she-smi}, 
Kastis, Kitson and McCarthy \cite{kas-kit-mcc} and Kastis and Power 
\cite{kas-pow-flexibility}.
 
Infinite bar-joint frameworks provide fundamental  mathematical models for material crystals and topological insulators (Dove et al \cite{dov-exotic}, Kane and Lubensky \cite{kan-lub}, Lubensky et al \cite{lub-et-al}, Rocklin et al \cite{roc-et-al}). The same is true for quasicrystals and the analysis of floppy modes (infinitesimal flexes) and zero modes (floppy modes with wave vectors in some sense). See for example, Stenull and Lubensky \cite{ste-lub} and Zhou et al \cite{zho-et-al}. In what follows we define \emph{linear zero mode spectra} for aperiodic bar-joint frameworks in $\bR^2$, that generalise certain linear structure in the crystallographic case, and these spectra are determined for quasicrystal parallelogram frameworks associated with regular multigrids.
 
The notion of an excitation mode wave vector for a quasicrystal framework $\G$ in $\bR^d$ is somewhat paradoxical since wave vectors are defined relative to a periodic structure for $\G$.
Indeed, for a crystallographic framework $\C$ one has a finite building block, or motif, of joints and bars, whose translates relative to a basis $\ul{a}$, generate $\C$. A zero mode $u$ is then a velocity field, satisfying the first-order flex conditions, that is determined by a finite velocity field $u_{\rm motif}$ on the motif joints together with a wave vector ${\bf k}$, or equivalently, in the complex case, by a unimodular phase-factor $\omega$ in the $d$-torus $\bT^d$. This infinitesimal flex definition of a zero mode is in analogy with Bloch's theorem in condensed matter physics. Since the phase-modulation of the velocities of $u_{\rm motif}$ within a translated block is constant we can view the entire phase modulation in terms of an \emph{ambient phase field} $\phi_\omega(x), x \in \bR^d,$ which is constant on the cells of a partition associated with the lattice for $\ul{a}$. In other words the zero mode $u$, with phase factor $\omega$ relative to the periodicity basis $\ul{a}$, is given as a pointwise product 
$
u = \phi_\omega \cdot \tilde{u}_{\rm motif},
$
where $\tilde{u}_{\rm motif}$ is the velocity field given by the periodic extension of ${u}_{\rm motif}$. That is, for each joint $p_i$,
\[
u(p_i) = (\phi_\omega\cdot \tilde{u}_{\rm motif})(p_i) = \phi_\omega(p_i)\tilde{u}_{\rm motif}(p_i).
\]

We generalise this particular mathematical formalism to aperiodic bar-joint frameworks $\G$ by considering phase fields associated with \emph{variable} vector space bases that are not necessarily commensurate with a fixed reference basis $\ul{a}$. 
Specifically, we consider how the presence of infinitesimal flexes which are approximately phase-periodic for specific directions can lead, in the limit, to the identification of lines in the reciprocal space for $\ul{a}$. A totality of such lines is considered as a \emph{linear zero mode spectrum}. 
In particular we determine such spectra for quasicrystal frameworks that are associated with regular multigrid parallelogram tilings, examples of which are the Penrose rhomb tilings \cite{deB} and the Ammann-Beenker tilings  \cite{bee}, from pentagrids and tetragrids respectively. 

In Section \ref{s:crystallographic} we give a self-contained account of the zero mode spectrum of a crystallographic bar-joint framework (or crystal framework) $\C$ in $\bR^d$. Moreover, we show in Theorem \ref{t:bandconverseDim2} that lines in the (unreduced) zero mode spectrum ${\bf K}(\C,\ul{a})$ \emph{necessarily} arise from hyperplane localised infinitesimal flexes. This seems to be a new observation in the mathematical theory and answers a question posed in Remark 4.12 of Badri, Kitson and Power \cite{bad-kit-pow-2}. The converse direction is a well-known paradigm in crystallography. That is, crystal structure symmetries can lead to linearly localised modes, or even modes with finite support, and these modes are observable experimentally, or in simulations, as spectral lines or planes. We remark that a connection between linear structure in the RUM spectrum and the presence of certain \emph{free bases} of infinitesimal flexes is examined in \cite{bad-kit-pow-2}.

In Section \ref{s:linefigures} we use the terminology of \emph{line figures} to summarise the results in Section \ref{s:crystallographic} for crystal frameworks in $\bR^2$. In particular the line figure of the zero mode spectrum, denoted $LF({\bf K}(\C,\ul{a}))$, is the set of lines through the origin that are parallel to a line of ${\bf K}(\C,\ul{a})$.

Linear zero mode spectra are defined in Section \ref{s:aperiodic} in the setting of Delone bar-joint frameworks in the plane. The simplest of these is the \emph{slippage spectrum} ${\bf L}_{\rm slip}(\G,\ul{a})$. This is a subset of the reciprocal space $\bR^2_{\bf k}$, for the reference basis $\ul{a}$ in $\bR^2$, consisting of a set of lines through the origin associated with certain relatively dense sets of linearly localised flexes. The ``slippage" terminology reflects the fact that  these localised flexes are restrictions of translation velocity fields. A Penrose rhomb tiling has slippage spectrum consisting of 5 lines through the origin. More generally, the framework $\G_P$ of a general regular multigrid parallelogram tiling $P$ has slippage spectrum given by the reciprocal line figure $RF(P)^{\ul{a}}$ of the \emph{ribbon figure} $RF(P)$, where the ribbon figure records the finite number of linear directions of the parallelogram ribbons of $P$.
See Theorem \ref{t:slippageSpecandRibbonFig} and Example \ref{e:penroseslippage}. 
Also, in the periodic case, when $\G_P$ is periodic for $\ul{a}$, 
the slippage spectrum ${\bf L}_{\rm slip}(\G_P,\ul{a})$
coincides with line figure $LF({\bf K}(\G_P,\ul{a}))$  (Theorem \ref{t:slippageEqualsRUMforG_P}). 

For these identifications  we use the characterisation of infinitesimal flexes of parallelogram frameworks obtained
in our companion article \cite{pow-qc1}, where we also give an explicit formula for $RF(P)$ in terms of the tile geometry of $P$.

The \emph{limit spectrum} ${\bf L}_{\rm lim}(\G,\ul{a})$ is a larger line figure in reciprocal space, which is similarly defined but in terms of general approximating phase-periodic velocity fields without any translational restriction. For the parallelogram frameworks $\G_P$ it agrees with the slippage spectrum. On the other hand, for the crystallographic kagome framework $\C_{\rm kag}$ in $\bR^2$, with periodicity basis $\ul{a}$, the slippage spectrum is empty whereas ${\bf L}_{\rm lim}(\C_{\rm kag},\ul{a})$ is 
the union of 3 lines through the origin and agrees with ${\bf K}(\C_{\rm kag},\ul{a})$.

In the final section we give further commentary and indicate some natural directions for investigation.

\section{Zero mode spectra for crystals}\label{s:crystallographic}
The existence of a wave vector {\bf k} for a zero mode excitation of  crystallographic bar-joint framework $\C$ means that the oscillatory motion $p_i(t)$ of a joint $p_i$ is determined by the motion of the joints in some fixed base unit cell, associated with a periodicity basis $\ul{a}$ for $\C$, together with {\bf k} and the integral coordinates labelling the cell containing $p_i$. 
In fact such a real-valued zero mode corresponds to a phase-periodic complex-valued infinitesimal flex of the framework. Specifically, the real part of the infinitesimal flex is a velocity field on the joints giving the initial velocity of the motion of the joints \cite{pow-poly}. 
With this perspective, of linearisation and complexification, the \emph{(reduced) RUM spectrum} $\Omega(\C,\ul{a})$ of a crystallographic bar-joint framework (or \emph{crystal framework} for brevity) may be defined quite directly. It is 
the subset of the $d$-torus $\bT^d$ consisting of the multiphases $\omega=(\omega_1,\dots ,\omega_d)$
of infinitesimal flexes which are periodic  with respect to $\ul{a}$ modulo the \emph{multiphase factor} $\omega$. 
 
To indicate this explicitly, assume that $d=2$ and let $\C=(G,p)$ be a crystallographic bar-joint framework with a periodicity basis $\ul{a}=\{a_1, a_2\}$ and an associated labelling of the joints, 
\[
p(V)=\{p_{\kappa,(i,j)}: (i,j)\in \bZ^2, 1 \leq \kappa \leq n\},
\]
 so that $p_{\kappa,(i,j)}= p_{\kappa,(0,0)}+ia_1+ja_2$. 
Here $G=(V,E)$ is the underlying structure graph and $n$ is the number of orbits of the joints under translations from the lattice of vectors
determined by $\ul{a}$. Let
$\V(\C,\bC)$ be the vector space of complex-valued velocity fields on the set of joints which we may identify with the vector space of sequences  $u:\{1,\dots , n\} \times \bZ^2\to \bC^2$.  Then for each $\omega \in \bT^2$ there is a finite-dimensional subspace of velocity fields $u$ which are 
\emph{phase-periodic} (or periodic-modulo-phase, or $\omega$-periodic), in the sense that $u_{\kappa,(i,j)}= \omega_1^i\omega_2^ju_{\kappa,(0,0)}$ for all $\kappa, i, j$.

A complex \emph{infinitesimal flex} of a bar-joint framework $(G,p)$ in $\bR^2$ is a velocity field $u:p(V) \to \bC^2$ which satisfies the first order flex condition for every bar. This means that 
\[
\langle u(p(v))-u(p(w)), p(v)-p(w)\rangle=0, \quad \mbox{ for } vw \in E.
\]
For a crystal framework $\C$ these flexes are identified with a subspace $\F(\C,\bC) \subseteq \V(\C,\bC)$.
Let us refer to an infinitesimal flex of $\C$ as an \emph{infinitesimal flex mode (IFM)} for $\ul{a}$ if it is nonzero and is phase-periodic for some multiphase $\omega$ in $\bT^2$. In particular an IFM is bounded. The \emph{RUM spectrum} $\Omega(\C, \ul{a})$ is defined to be the set  of the multiphases $\omega$ for these IFMs. 

For general dimension $d$,
let $\omega=(\omega_1,\dots ,\omega_d) \in \bT^d, \omega_i=e^{2\pi i\gamma_i},$ and let $\omega^{k}$ be the product 
$\omega_1^{k_1}\cdots \omega_d^{k_d}$ for $k\in \bZ^d$.
For a crystal framework $\C$ in $\bR^d$ with periodicity basis $\ul{a}$, the (unreduced) \emph{zero mode wave vector spectrum} ${\bf K}(\C,\ul{a})$ is defined to be the set of \emph{wave vectors}  ${\bf k}=(\gamma_1, \dots , \gamma_d)$ such that
there is an IFM for $\omega$. This means that $u$ is phase-periodic for $\omega$ in the sense that for the shift isometries
\[
T_k: (x_1,\dots , x_d) \to (x_1,\dots , x_d) + (k_1a_1+\dots + k_da_d), \quad k\in \bZ^d,
\]
we have 
$
u(T_k(p_j))= \omega^{k}
u(p_j)
$ for each joint $p_j$. 
The wave vectors ${(\gamma_1, \dots , \gamma_d)}$ are viewed as  elements of the space $\bR^d_{\bf k}$, with its standard basis, $\ul{b}=\{b_1,\dots ,b_d\}$ and so 
${(\gamma_1, \dots , \gamma_d)}$ is identical to the vector ${\gamma_1b_1+ \dots + \gamma_db_d}$ in $\bR^d_{\bf k}$. We refer to  $\bR^d_{\bf k}$ as the \emph{reciprocal space for} $\ul{a}$ since
this terminology conforms with the usual usage in crystallography. That is, the basis satisfies the identities $\langle a_i,b_j\rangle =\delta_{ij}$.  In particular it may also be considered as as the usual dual vector space of the coefficient space of vectors $(s_1, \dots , s_d)$ for the basis $\ul{a}$.

The image of ${\bf K}(\C,\ul{a})$ in $[0,1)^d$ under the quotient map $\bR^d_{\bf k}\to \bR^d_{\bf k}/\bZ^d$
is the set of \emph{reduced} wave vectors, and this is the convenient wave vector form, or logarithmic form, of the RUM spectrum 
$\Omega(\C,\ul{a})$ 
used by crystallographers which we may denote as $\Omega^{\rm log}(\C,\ul{a})$. From the definitions it is evident that ${\bf K}(\C,\ul{a})$ is the periodic extension of $\Omega^{\rm log}(\C,\ul{a})$. 
Although there is a complete equivalence between the wave vector and multiphase formalism it is conceptually convenient to consider both forms. Also, in Section \ref{s:aperiodic} we consider variant reduced and unreduced spectra for aperiodic frameworks.

The RUM multiphases $\omega$ are given as the solutions of a set of multivariable polynomial equations. Thus
$\Omega(\C, \ul{a})$ is a compact subset of $\bT^d$ and is also a real algebraic set in its wave vector representation in $[0,1)^d$.  In dimension $d$ there are therefore $d+1$ possible values for the topological or Hausdorff dimension of  $\Omega(\C, \ul{a})$. In view of Proposition \ref{p:surjection} below this value is independent of the choice of periodicity basis and we refer to it as the \emph{RUM dimension}, $\dim_{\rm rum}(\C)$, of $\C$
\cite{owe-pow-crystal}, \cite{pow-poly}. 
A gallery of examples with RUM spectra of different dimension is given in Badri, Kitson and Power \cite{bad-kit-pow-1}. Note that an infinitesimal translation velocity field is an IFM for 
$\ul{1}=(1,1,\dots ,1)$ and so the origin is present in any RUM spectrum $\Omega(\C, \ul{a})$.

For $k=(k_1, \dots , k_d)$ in $\bZ^d$  let $k\cdot\ul{a}$ be the periodicity basis
$\{k_1a_1,\dots , k_da_d\}$. Then there is a natural map  ${\bf K}(\C, \ul{a}) \to {\bf K}(\C, k\cdot\ul{a})$ 
given by 
\[
(\gamma_1,\dots ,\gamma_d)
\to (k_1\gamma_1,\dots , k_d\gamma_d)
\]
and a corresponding map for the RUM spectrum quotients.
This follows since an infinitesimal flex mode $u$, for the pair $(\omega, \ul{a})$, is an infinitesimal flex mode for the pair $(\omega^{(k)}, k\cdot\ul{a})$ where 
$\omega^{(k)} = (\omega_1^{k_1},\dots , \omega_d^{k_d}).$

\begin{prop}\label{p:surjection} The maps  ${\bf K}(\C, \ul{a}) \to {\bf K}(\C, k\cdot\ul{a})$ and  $\Omega(\C, \ul{a}) \to \Omega(\C, k\cdot\ul{a})$
 are surjections.
\end{prop}

To see this one must show that if $u'$ is an IFM for the pair $(\eta, k\cdot\ul{a})$ then there exist a choice of complex roots $\omega_i = \eta_i^{1/k_i}$, for $ 1\leq i \leq d$, such that there is an IFM $u$ with multiphase $\omega=(\omega_1, \dots , \omega_d)$ for the periodicity basis $\ul{a}$. This is a consequence of elementary representation theory for finite abelian groups. (See the appendix of \cite{pow-poly}.)
\medskip
 
Of particular interest is the identification of the RUM spectrum when $\C$ is a \emph{Maxwell framework} in $\bR^d$, that is, one for which the average coordination (valency of the joints) is equal to $2d$. 
Two basic examples in 2 dimensions are the square grid framework, which we denote as $\C_{\bZ^2}$, and the well-known kagome framework $\C_{\rm kag}$. It can be shown that for any periodicity basis $\ul{a}$ the wave vector spectrum ${\bf K}(\C_{\bZ^2}, \ul{a})$ (resp. ${\bf K}(\C_{\rm kag}, \ul{a})$) is the union of periodic extensions of 2 (resp. 3) straight lines through the origin together with their integral translates and so these frameworks have RUM dimension 1. 

For the case of maximal RUM dimension, dimension $d$, we have the following characterisation. A \emph{local infinitesimal flex} is one which is finitely supported.

\begin{thm}\label{t:orderN}\cite{pow-poly}
The following properties are equivalent for a crystal framework in $\bR^d$ with periodicity basis $\ul{a}$.

(i) $\C$ has a local infinitesimal flex.

(ii) $\Omega(\C, \ul{a}) = \bT^d$.

(iii) $\Omega(\C, \ul{a})$ has dimension $d$.
\end{thm}

\begin{proof}
 The equivalence of (ii) and (iii) holds since the RUM spectrum is a real algebraic variety in $\bR^{2d}$.
To see that (i) implies (ii) let $z$ be a local infinitesimal flex of $\C$, where $\ul{a}$ is a periodicity basis for $\C$. We must construct, for any given multiphase $\omega$, an IFM $u$ for the pair $(\ul{a}, \omega)$. 
Consider the translated flexes $T_kz$, for $k\in \bZ^d$, defined by $T_kz(p_i)= z(T_{-k}(p_i))$. Note that there is an upper bound, $M$ say, such that for any joint $p_i$  the number of the translated flexes which have  $p_i$ in their supports is no greater than $M$. Thus we may define the velocity field 
\[
u= \sum_{k\in \bZ^d} \omega^{-k} T_kz
\]
and this is also an infinitesimal flex. Moreover,
it is an IFM for  $(\ul{a}, \omega)$. The converse direction, (ii) implies (i), is more involved and follows from Lemma \ref{l:bandconverseOrderN}.\end{proof}

\subsection{Localised flexes imply lines of wave vectors}\label{ss:bandlimited}\label{s:bandlimited}
Let $H$ be a line in $\bR^2$ through the origin.
A subset of joints of a bar-joint framework $\G$ in $\bR^2$ is  \emph{$H$-localised} if there is an upper bound to their distance from $H$. A velocity field or infinitesimal flex of $\G$ is \emph{linearly localised} if its support is {$H$-localised} for some line $H$.

Let $\C$ be a crystal framework in $\bR^2$ with
periodicity basis $\ul{a}= \{a_1, a_2\}$.
A line $H$ through the origin is said to be \emph{rational} for $\C$ if 
it is parallel to some vector $j_1a_1+j_2a_2$, where $j_1, j_2$ are integers. This is a well-defined notion since any
two periodicity bases are equivalent by an invertible matrix with rational entries.  

An $H$-localised infinitesimal flex $u_{\rm loc}$, for a rational line $H$ for $\C$, is said to be \emph{periodic}, or \emph{rationally periodic}, if it is periodic relative to some integral direction vector $j_1a_1+j_2a_2$ for $H$. In particular, $H= \bR(j_1a_1+j_2a_2)$. Also, $u_{\rm loc}$ is  \emph{phase-periodic} if it is periodic up to a multiplicative unimodular factor $\lambda$.
 
\begin{prop}\label{p:dualline}
 Let $\C$ be a crystal framework in $\bR^2$ with periodicity basis $\ul{a}=\{a_1, a_2\}$ and let $z$ be a nonzero $H$-localised infinitesimal flex for the rational line $H=\bR a_1$ which is phase-periodic, with respect to translation by $a_1$,  with unimodular phase factor $\lambda_1=e^{2\pi i\gamma_1}$. Then the zero mode spectrum {\bf K}$(\C,\ul{a})$  contains the line $\{(\gamma_1, t):t\in \bR\}$.
\end{prop}

\begin{proof} As in earlier notation, let $T_k, k \in \bZ^2$, be the translation group for $\ul{a}$.
For each joint $p_j$ we have $z(T_{(k_1,0)}p_j)=\lambda_1^{k_1} z(p_j)$.
For $\lambda_2 \in \bT$ the velocity field
\[
u= \sum_{k_2\in \bZ} \lambda_2^{-k_2} T_{(0,k_2)}z
\]
is well-defined, since $z$ being $H$-localised ensures that each joint lies in the support of only finitely many of the velocity fields $T_{(0,k_2)}z$, for $k_2$ in $\bZ$. In particular $u$ is an infinitesimal flex. Moreover it is an IFM for the multiphase $(\lambda_1,\lambda_2)$ relative to the periodicity basis $\ul{a}$. Since $\lambda_2=e^{2\pi it}$ is arbitrary, {\bf K}$(\C,\ul{a})$ contains $\{(\gamma_1,t), t \in \bR\}$.
\end{proof}

From a geometric point of view note that the spectral line in the previous proposition is the line orthogonal to the line $s_2=0$ in the (new) coefficient Euclidean space of pairs $(s_1, s_2)$ representing points $s_1a_1+s_2a_2$ in the ambient Euclidean space.

Let $\ul{a}$ be a periodicity basis for $\C$ in $\bR^2$ as before and consider now the line $H=\bR a_1'$ in the ambient space for $\C$ given by a rational vector $a_1'$. Suppose moreover that $u_{\rm loc}$ is a nonzero $H$-localised phase-periodic flex with phase-factor $e^{2\pi i \gamma_1'}$ for translation by  $a_1'$ in $\bR^2$. 
Then $u_{\rm loc}$ determines a line $L'$ in {\bf K}$(\C,\ul{a})$.
The line  $L$ parallel to $L'$ and containing the origin is denoted
$H^{\ul{a}}$, and is not dependent on the phase factor.  
In general, if $a_1'=\alpha_1a_1 + \beta_1a_2$ then in the coefficient space for $\ul{a}$ the line $\bR a_1'$ is represented as $\bR(\alpha_1, \beta_1)$ and the spectral line $H^{\ul{a}}$ is the orthogonal line $\bR (\beta_1, -\alpha_1)$.

With this notation we have the following corollary. In the next section we see that it has a converse. 
 
\begin{cor}\label{c:dualline}
Let $\C$ be a crystal framework in $\bR^2$ with periodicity basis $\ul{a}$ and let
 $H$ be  a rational line through the origin for which there exists a nonzero $H$-localised phase-periodic infinitesimal flex.
Then {\bf K}$(\C,\ul{a})$ contains a line parallel to 
$H^{\ul{a}}$.
\end{cor}


Let us also note that in two dimensions the wave vector spectrum {\bf K}$(\C,\ul{a})$ contains a line with nonrational gradient only in the extreme case that it is equal to $\bR^2$. 
Indeed, in this irrational case the reduced RUM spectrum is dense in $[0,1)^2$. Since it is closed set in $[0,1)^2$, by the compactness of $\Omega(\C,\ul{a})$, it is equal to $[0,1)^2$, and so {\bf K}$(\C,a) = \bR^2$.



\begin{example}\label{e:unboundedflex}  A crystal framework  may have a linearly localised infinitesimal flex and yet have a trivial RUM spectrum. 
For example, Figure \ref{f:unboundedonly2Dcrystal} indicates a crystal framework $\C_{\bZ^2}^{++}$ which is a double augmentation of the grid framework $\C_{\bZ^2}$ whose joints have integer coordinates. 
\begin{center}
\begin{figure}[ht]
\centering
\includegraphics[width=8cm]{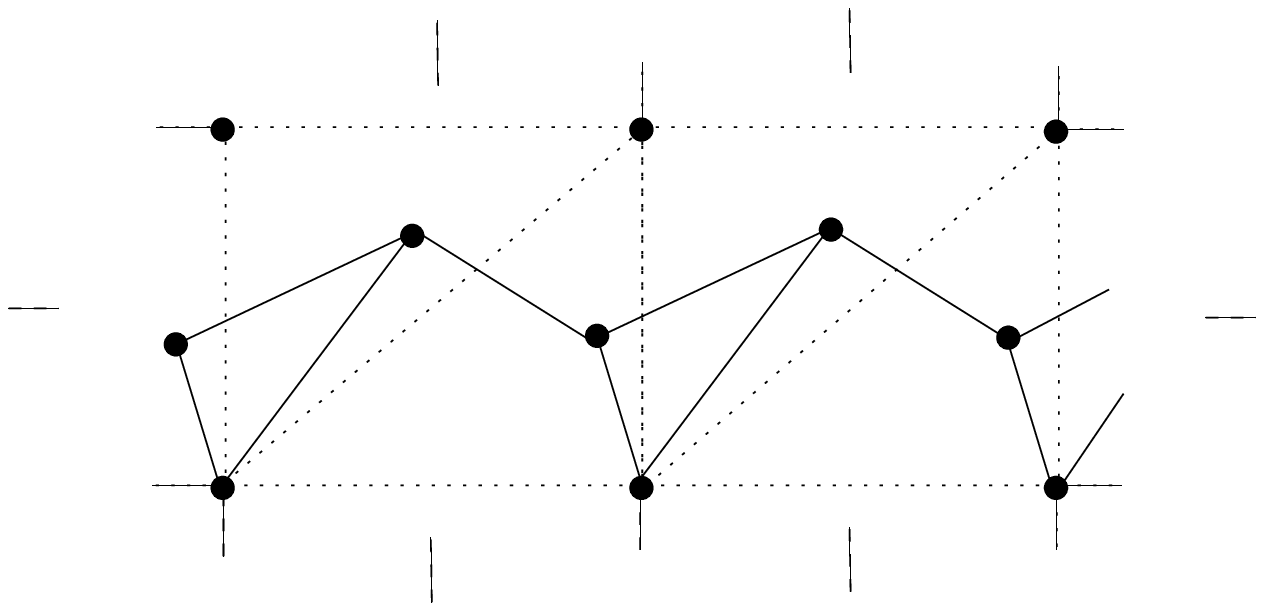}
\caption{The infinitesimally flexible crystal framework $\C_{\bZ^2}^{++}$ in $\bR^2$ with trivial RUM spectrum $\{(0,0)\}$ in $[0,1)^2$.}
\label{f:unboundedonly2Dcrystal}
\end{figure}
\end{center}
In the first augmentation diagonal bars have been added creating an infinitesimally rigid framework $\C_{\bZ^2}^+$. Then each row of joints of $\C_{\bZ^2}$ has been augmented by pairwise connected triangle frameworks. The resulting crystal framework has nonrigid motion infinitesimal flexes but they are necessarily unbounded.
Indeed, each horizontal chain of added joints is the support set of nonzero linearly localised infinitesimal flexes where the individual  velocity vectors increase geometrically in magnitude in the positive $x$-direction.
\end{example}

\begin{rem}\label{e:kastispower}
In Kastis and Power \cite{kas-pow-synthesis}, \cite{kas-pow-flexibility} we use techniques from algebraic spectral synthesis to obtained a complete characterisation of the infinitesimal rigidity of a crystal framework $\C$ in  terms of a \emph{geometric spectrum} $\Gamma(\C,a)$ in $\bC^d\backslash \{0\}$. This spectrum accommodates  modes and flexes with geometric growth, such as those existing  in our Example \ref{e:unboundedflex}, and it extends the RUM spectrum in $\bT^d$. It is potentially significant for the identification of bounded modes of semi-infinite crystals relative to a free surface of $\C$. See also, Power \cite{pow-seville}. On the other hand further spectral synthesis methods are required to characterise when there are no proper \emph{bounded} infinitesimal flexes of $\C$ beyond translations. The case of almost periodic infinitesimal flexes was resolved in Badri, Kitson and Power \cite{bad-kit-pow-1} with the RUM spectrum playing a role analogous to that of the Bohr spectrum of an almost periodic sequence.
\end{rem}

\subsection{Lines of wave vectors imply localised flexes}\label{ss:lines_imply_localised}
We now obtain, in Theorem \ref{t:bandconverseDim2}, a converse direction to Proposition \ref{p:dualline}. We first remark that this proposition  extends readily to dimensions $d\geq 3$. The terminology in this case is that a \emph{rational hyperplane} $H$ for $\C$ in $\bR^d$ is a hyperplane which is spanned by $d-1$ vectors in a periodicity basis $\ul{a}$. Since any two periodicity bases for $\C$ are equivalent by a matrix in $GL(\bR^d,\bQ)$ this is well-defined. In this case an $H$-localised velocity field $u$ is said to be \emph{phase-periodic} for a phase factor $\omega$ in $\bT^{d-1}$ if for some such basis $\{a_1,\dots , a_{d-1}\}$ for $H$, we have $T_ku =\ol{\omega}^ku$ for all $k\in \bZ^{d-1}$.   

\begin{lemma}\label{l:bandconverse}
Let $\C$ be a crystal framework in $\bR^d, d\geq 2$, and suppose that there is a periodicity basis $\ul{a}= \{a_1, \dots ,a_{d}\}$ such that {\bf K}$(\C, \ul{a})$ contains the line of points 
$\{(\gamma_1,\dots ,\gamma_{d-1},t):t\in \bR\}$. Then there exists an $H$-localised phase-periodic flex, for the hyperplane $H$ spanned by $\{a_1, \dots ,a_{d-1}\}$, with phase factor $\omega = (e^{2\pi i\gamma_1}, \dots , e^{2\pi i\gamma_{d-1}})$.
\end{lemma}

\begin{proof} Suppose first that $d=2$. 
Then the given line is $\{(\gamma_1,t):t\in \bR\}$ and $H=\bR a_1$. Let $Q$ be the set of $m$ rational numbers $l/{m}$ with $ 0\leq l<m$. Then for each $q \in Q$ there is an IFM $u_q$ for the wave vector $(\gamma_1,l/m)$.  These are linearly independent and span an $m$-dimensional vector space of infinitesimal flexes, denoted $\F_m$.

Consider the parallelogram
$R_m$ of points with positions $t_1a_1+t_2a_2$ in the ambient space with $0\leq t_1 <1, 0\leq t_2 <m$ . This ``vertical" rectangle is associated with $1$-fold phase periodicity in the ``horizontal" direction $a_1$, and $m$-fold periodicity in the direction $a_2$. Let $\V_m$ be the vector space of all velocity fields which, are 1-fold phase-periodic with phase factor $e^{2\pi i\gamma_1}$, with respect to $a_1$, and $m$-fold periodic for the vector $a_2$. We refer to this loosely as $R_m$-periodicity. In particular $\F_m\subseteq \V_m$.

Let  $\tilde{R}_m$ be the horizontal $a_1$-periodic band, generated by $R_m$ and let $J_m$ be the set of joints 
belonging to ``overlapping" bars which have exactly one joint in $R_m$ and one joint outside $\tilde{R}_m$, as in Figure \ref{f:overlappingbars}.
Note that the size of $J_m$ is independent of $m$. Thus, the space $\J_m$ of $R_m$-periodic velocity fields with support contained in the $a_1$-periodic extension $\tilde{J}_m$ of $J_m$ has a fixed dimension.
It follows that we may choose $m$ large enough so that the natural restriction map from $\F_m$ to $\J_m$ has  nontrivial kernel, containing a nonzero $R_m$-periodic infinitesimal flex $u$. Since $u$ assigns the zero velocity vector to both joints for any bar with a joint in the complement of $\tilde{R}_m$, it follows that the linearly localised velocity field $z$ which is defined to be zero outside $\tilde{R}_m$ and equal to $u$ on joints in $\tilde{R}_m$ is in fact an infinitesimal flex. The flex $z$ is $H$-localised and phase-periodic for $e^{2\pi i\gamma_1}$, completing the proof in this case.

\begin{center}
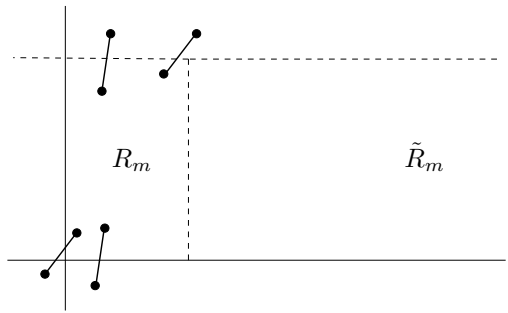
\begin{figure}[ht]
\centering
\input{OverlappingbarsLATEX.pstex_t}
\caption{Bars in $R_m$ overlapping the boundary of $\tilde{R}_m$.
} 
\label{f:overlappingbars}
\end{figure}
\end{center}
 
For general $d$ the argument is the same. Take $R_m$ to be the parallelepiped of points with positions $t_1a_1+\dots + t_da_d$, with $t_i \in [0,1)$ for $1\leq t_i\leq d-1$, and $t_d \in [0,m)$. Let $\tilde{R}_m$ be the $\{a_1, \dots , a_{d-1}\}$-periodic set generated by $R_m$.
The essential point, once again, is that the set $J_m$, defined as before, has fixed size.
\end{proof}

It follows that we have the following converse to Corollary \ref{c:dualline}.

\begin{thm}\label{t:bandconverseDim2}
Let $\C$ be a crystal framework in $\bR^2$ with periodicity basis $\ul{a}$ and let $H$ be a rational line in $\bR^2$ for $\C$ with reciprocal line $H^{\ul{a}}$ in $\bR^2_{\bf k}$. 
If {\bf K}$(\C, \ul{a})$ contains a line parallel to  $H^{\ul{a}}$ then $\C$ has a nonzero $H$-localised phase-periodic flex.
\end{thm}

Lemma \ref{l:bandconverse} generalises in a routine way to give the next lemma showing that a linear subspace of dimension $r<d$ in phase space implies the existence of $H$-localised phase-periodic flexes for an associated  linear subspace $H$ in ambient space of dimension $d-r$.

\begin{lemma}\label{l:bandconverse2}
Let $\C$ be a crystal framework in the space $\bR^d, d\geq 2$, and suppose that there is a periodicity basis $\ul{a}$ such that {\bf K}$(\C, \ul{a})$ contains the linear subspace
$\{(\gamma_1,\dots ,\gamma_{d-r},t_1,\dots ,t_r):t_i\in \bR\}$ for some $1\leq r <d$. Then there exists an $H$-localised phase-periodic flex, for the ambient linear subspace $H$ spanned by $\{a_1, \dots ,a_{d-r}\}$, with phase factor $(e^{2\pi i\gamma_1}, \dots , e^{2\pi i\gamma_{d-r}})$.
\end{lemma}

Also, in the extreme $r=d$ case {\bf K}$(\C, \ul{a})=\bR^d_{\bf k}$ we have the following.

\begin{lemma}\label{l:bandconverseOrderN}
Let $\C$ be a crystal framework in the space $\bR^d, d\geq 2$, with periodicity basis $\ul{a}$. If {\bf K}$(\C, \ul{a})=\bR^d_{\bf k}$ then there exists a nonzero local infinitesimal flex.
\end{lemma}

\begin{proof} The following argument, for $d=2$, is similar to the proof of  Lemma \ref{l:bandconverse}, and the argument for general $d$ is entirely similar. Let $Q$ be the set of $m^2$ points in the RUM spectrum of the form $(k/m,l/m),$ for $0\leq k,l\leq m-1$, with associated IFMs $u_q, q \in Q$. Let $P_m$ be the parallelogram in ambient space of points with positions $t_1a_1+ t_2a_2,$ with $0\leq t_1, t_2 \leq m$, and let $\V_m$ be the space of velocity fields which are periodic for the vectors $ma_1, ma_2$. 
The subspace of periodic velocity fields in $\V_m$ which are supported on joints for bars that overlap the boundary of $P_m$ has dimension of order $m$. On the other hand the linear span, $\F_m$ say, of the IFMs $u_q, q \in Q,$ has dimension $m^2$, and $\F_m\subseteq \V_m$. It follows that there exist a nonzero linear combination $u$ of the $u_q$ which is zero on the joints of the overlapping bars. Thus, the restriction of $u$ to joints in $P_m$ defines a local infinitesimal flex with support in $P_m$.
\end{proof}


\section{Line figures for localised flexes}\label{s:linefigures}
In the previous section we have given connections between linearly localised infinitesimal flexes and linear structure in the RUM spectrum. As a prelude to considering such relationships for quasicrystal bar-joint frameworks we summarise these connections for crystal frameworks in $\bR^2$ and introduce the additional terminology of {line figures}.

A \emph{line figure} in the ambient space $\bR^d$, or in the reciprocal space of a basis $\ul{a}$ in $\bR^d$, is a set of lines through the origin. 
For a general (Borel) set $\M$ in $\bR^d$ the \emph{line figure} $LF(\M)$ is the union of the set of lines through the origin which are parallel to a line in $\M$.

\begin{definition}\label{d:LFF}
Let $\G$ be a countable bar-joint framework in $\bR^d$, for $d\geq 2$. Then the \emph{localised flex figure} of $\G$, denoted  $LFF(\G)$, is the line figure in $\bR^d$ formed by lines $H$ through the origin for which there exists a nonzero $H$-localised infinitesimal flex. 
\end{definition}

\begin{definition}\label{d:PPLFF}
Let $\C$ be a crystal framework in $\bR^2$ with a periodicity basis $\ul{a}$.
Then the \emph{localised phase-periodic flex figure} $LPFF(\C)$ is the line figure in $\bR^2$ formed by the rational lines $H$ through the origin for which there exists a nonzero phase-periodic $H$-localised infinitesimal flex. 
\end{definition}

In particular the localised phase-periodic flex figure is only defined for crystal frameworks and it consists of at most a countable set of lines through the origin.

From the compactness of the RUM spectrum we have noted (in the comments following Corollary \ref{c:dualline}) that for ambient dimension $d=2$ we have the following dichotomy. Either $LF({\bf K}(\C, \ul{a}))$ is equal to $\bR^2_{\bf k}$ or  it consists of a (possibly empty) set of lines with rational slope. In the latter case, when the spectrum is proper, it follows from Theorem \ref{t:bandconverseDim2} that a line $L$ in $LF({\bf K}(\C, \ul{a}))$ is determined by a rational line $H$ in ambient space for which there exists a phase-periodic $H$-localised infinitesimal flex. The converse also holds, by Corollary \ref{c:dualline}. Thus, writing $H^{\ul{a}}$ for $L$, as we did prior to Corollary \ref{c:dualline}, we  have a bijective correspondence $\theta_{\ul{a}}:H \to H^{\ul{a}}$ between the ambient space lines $H$ of $LPFF(\C)$ (which are necessarily rational with respect to $\ul{a}$ if the spectrum is proper) and lines in the line figure $LF({\bf K}(\C, \ul{a}))$ (which are necessarily rational with respect to the reciprocal basis if the spectrum is proper). 
Thus we have the following equality.

\begin{thm}\label{t:LFFandLF}
Let $\C$ be a crystal framework in $\bR^2$ with periodicity basis $\ul{a}$ and proper zero mode spectrum ${\bf K}(\C,\ul{a})$. Then
\[
 LF({\bf K}(\C,\ul{a})) =  \theta_{\ul{a}}(LPFF(\C)).
\]
\end{thm}

Formally, the map $\theta_{\ul{a}}$ is a map $P(\bR^2)\to P(\bR^2_{\bf k})$ between projective spaces. However, we have seen from the remarks preceding Corollary \ref{c:dualline} both the geometric specification of $H^{\ul{a}}$ and how the map may be linearly implemented.
We use this map in Definition \ref{d:slippagespec} where it is considered, as above, as a map from a line figure in ambient space to a line figure in the reciprocal space of an ambient space basis $\ul{a}$.

The next lemma shows how the linear figure of the zero mode spectrum is  transformed under a change of periodicity basis. 

\begin{lemma}\label{l:changebasisformula}
Let $\C$ be a crystal framework in $\bR^2$ with periodicity basis $\ul{a}$ and let $\ul{a}^* = \{a_1^*,a_2^* \}$, with
$a_1^*=\alpha_1a_1+\beta_1a_2$ and $ a_2^* = \alpha_2a_1+\beta_2a_2$ where  $\alpha_1, \beta_1,\alpha_2,\beta_2$ are integers. Let
$ Z$ be the integral matrix
$\left[\begin{matrix}
\alpha_1&\beta_1\\
\alpha_2&\beta_2
\end{matrix}
\right].$ Then 
\medskip

(i) $ Z{\bf K}(\C, \ul{a})\subseteq {\bf K}(\C, \ul{a}^*)$.
\medskip

(ii)  $LF({\bf K}(\C, \ul{a}^*)) =  Z(LF({\bf K}(\C, \ul{a}))) = \theta_{{\ul{a}^*}}\theta_{\ul{a}}^{-1}(LF({\bf K}(\C, \ul{a})))$.
\medskip

Also, let $\ul{a}' = \{a_1', a_2' \}$ be a periodicity basis with $a_1'=\alpha_1'a_1+\beta_1'a_2$ and $a_2' = \alpha_2'a_1+\beta_2'a_2$, where $\alpha_1', \beta_1',\alpha_2',\beta_2'$ are rational numbers. If $ Z'$ is the rational matrix 
$\left[\begin{matrix}
\alpha_1'&\beta_1'\\
\alpha_2'&\beta_2'
\end{matrix}
\right]$ then 
\medskip

(iii) $LF({\bf K}(\C, \ul{a}')) =  Z'(LF({\bf K}(\C, \ul{a})))$.

\end{lemma} 

\begin{proof}
(i) Let $(\gamma_1, \gamma_2)\in {\bf K}(\C, \ul{a})$ with associated IFM $u$. Let $T^*_l, l\in \bZ^2$, be the translation isometries for the basis $\ul{a}^*$. Then
$
T^*_l = T_1^{\alpha_1l_1+\alpha_2l_2}T_2^{\beta_1l_1+\beta_2l_2}.
$
Thus, for a joint $p_{\kappa,k}$ of $\C$, with respect to the basis $\ul{a}$, where $1\leq \kappa\leq n, k\in \bZ^2$, we have
\[
u(T^*_lp_{\kappa, k})= e^{2\pi i(\alpha_1l_1+\alpha_2l_2)\gamma_1}
 e^{2\pi i(\beta_1l_1+\beta_2l_2)\gamma_2}u(p_{\kappa, k})
=e^{2\pi i\delta_1l_1}e^{2\pi i\delta_2l_2}u(p_{\kappa, k})
\]
where $\delta_i=\alpha_i\gamma_1+\beta_i\gamma_2$, for $i=1,2$. It follows that $(\delta_1,\delta_2)\in {\bf K}(\C, \ul{a}^*)$, as required.

(ii)  If ${\bf K}(\C, \ul{a})= \bR^2$ then equality holds. Suppose on the other hand that the zero mode spectrum is proper. Then it follows from Theorem \ref{t:LFFandLF} that
  $LF({\bf K}(\C, \ul{a}))$ and  $LF({\bf K}(\C, \ul{a}^*))$ are in bijective correspondence by the linear map $\theta_{{\ul{a}^*}}\theta_{\ul{a}}^{-1}$. On the other hand, using (i), we have 
\[
Z(LF({\bf K}(\C, \ul{a}))) = LF(Z{\bf K}(\C, \ul{a}))\subseteq LF({\bf K}(\C, \ul{a}^*).
\]
Putting these facts together it follows that the inclusion is an equality.

(iii) This follows from (ii) by considering a periodicity basis whose vectors are integral linear combinations of the vectors of $\ul{a}$, as well as being integral linear combinations of the vectors of $\ul{a'}$. 
\end{proof}

\section{Zero mode spectra for aperiodic frameworks}\label{s:aperiodic}
The zero modes, or infinitesimal flex modes (IFMs)
of a crystallographic framework $\C$ are determined by a finite data set, for a periodically repeating block of nodes and bonds, and a multiphase $\omega \in \bT^d$.  Nevertheless, zero modes capture many aspects of the flexibility of $\C$.
It is of interest then to determine analogous zero modes and spectra, or partial analogues, for quasicrystallographic frameworks. We shall do this in terms of generalised phase fields and infinitesimal flexes which are approximately phase-periodic. 

\subsection{Phase-periodic velocity fields on $\bR^d$}\label{ss:phasefieldsonRd}
We first define phase fields and phase-periodic velocity fields on $\bR^d$, with respect to partitions associated with a basis $\ul{a}$. Phase-periodic velocity fields can then be defined on a Delone bar-joint framework by restrictions. In particular we define the zero mode spectrum of a crystallographic framework in these terms. 

Recall that a Delone set in $\bR^d$ is  a countable well-separated set which is also relatively dense. The following definition seems to us to give the most natural catch-all aperiodic setting in which to define zero mode spectra with respect to some reference basis, with no a priori assumptions.

\begin{definition}\label{d:deloneframework} A \emph{Delone  bar-joint framework}  in $\bR^d$ is a countable bar-joint framework $\G=(G,p)$, with $G=(V,E)$ a simple countable graph, such that the set of joints $p(v)$, for $v$ in $V$, is a Delone set, and the set of bar lengths $|p(v)-p(w)|$, for $vw$ in $ E$, is uniformly bounded. 
\end{definition}

Let $\bR^d$, with $d\geq 2$, be viewed as an ambient Euclidean space,  with its standard basis and points $x$ with coordinates $x=(x_1,\dots , x_d)$, and let $\ul{a}=(a_1, \dots , a_d)$ be a second basis. 
The \emph{cell partition} for $\ul{a}$ of $\bR^d$ is defined to be the partition $\P = \{C_k:k\in \bZ^d\}$ where $C_k$ is the parallelepiped
\[
C_k = [k_1a_1,(k_1+1)a_1)\times \dots \times [k_da_d,(k_d+1)a_d).
\] 

 For notational simplicity let $d=2$. 
A 
\emph{phase field} for $\P$ is a map $\phi=\phi_{\omega,\ul{a}}$ from  $ \bR^2$ to $\bT^2$ determined by the multiphase $\omega =(\omega_1,\omega_2) \in \bT^2$, where for $(x,y)$ in the cell $C_k$
we have $
\phi_{\omega,\ul{a}}(x,y)=\omega_1^{k_1}\omega_2^{k_2}.
$
Also, for a velocity vector $b$ in $\bC^2$, define an associated discontinuous $\bC^2$-valued \emph{phase-periodic velocity field} $\phi_{\omega,\ul{a}}\otimes b$ on $\bR^2$, with
\[
(\phi_{\omega,\ul{a}}\otimes b)(x,y) =\omega_1^{k_1}\omega_2^{k_2} b, \quad \mbox{ for }(x,y)\in C_k. 
\]
Note, for example, that if $(\omega_1,\omega_2)=(1,\lambda_2)$ then 
the restriction of this velocity field to the band $\bR a_1 \times [0, a_2)$ acts as translation by $b$. On the parallel band obtained by translation by  $na_2, n\in \bZ$, it acts as a constant velocity field given by the velocity vector $\lambda_2^nb$.

More generally, we define the matricial variant, $\phi_{\omega,\ul{a}}\otimes B$, for a matrix of vectors $B=(b_{l,m})$, for
$0\leq l \leq L-1, 0\leq m \leq M-1$. This is the velocity field on $\bR^2$ which assigns the velocity vector $b_{lm}$ to the $(l,m)$ subcell of a cell partition of $C_{(0,0)}$, and which is defined on the other cells $C_k$ and their corresponding subcell partitions, by phase-periodic extension. 
We refer to the map $\phi_{\omega,\ul{a}}\otimes B$ as a 
\emph{phase-periodic velocity field} 
for the triple $(\ul{a},L,M)$, and we refer to $B$ as the \emph{unit cell velocity vector matrix}. 

For a Delone bar-joint framework $\G$ in $\bR^2$ one can choose $L,M$ large enough so that each subcell contains at most one joint. In particular, for a crystal framework $\C$, with a specified periodicity basis  $\ul{a}$, it follows that every phase-periodic velocity field $u$, associated with the pair $({\omega}, \ul{a}),$ is the restriction of a matricial phase field $\phi_{\omega,\ul{a}}\otimes B$ 
to the joints $p_i$ of $\C$. With these assumptions, with fixed $L, M$,
we have the following formulation of the zero mode spectrum for the pair $\C, \ul{a}$;
\[
\Omega(\C,\ul{a}) = \{\omega\in\bT^2 : \exists u\in \F(\C,\bC)\backslash\{0\} \mbox{ and } B  \mbox{ with } u(p_i) =(\phi_{\ol{\omega},\ul{a}}\otimes B)(p_i), \forall p_i\},
\]
\[
{\bf K}(\C,\ul{a}) = \{(\gamma_1,\gamma_2)\in \bR^2:
\omega = (e^{2\pi \gamma_1i},e^{2\pi \gamma_2 i}) \in  \Omega(\C,\ul{a})\}.
\]

One can define the RUM spectrum of crystal frameworks in general dimensions in exactly the same way; for $d\geq 3$, the matricial phase field  $\phi_{\omega,\ul{a}}\otimes B$ is given by an array $B$ of multi-indexed vectors $b_m$, with the coordinates $m_i$ of $m$ occurring in the range
$0 \leq m_i\leq M_i-1$ according to a subcell partition of the unit cell $C_{(0,\dots,0)}$ of the $d$-fold partition for $\ul{a}$.

\subsection{Parallelogram frameworks and their phase fields}
\label{ss:G_Pandphasefields}
When $\ul{a}$ is not a periodicity basis, or when $\G$ is a general Delone framework, one can similarly define the sets
${\bf K}(\C,\ul{a})$, ${\bf K}(\G,\ul{a})$.
However, such strict analogues to the crystallographic definitions yield only limited multiphase data that is associated with bases that are \emph{commensurate} with $\ul{a}$ in the sense that they are equivalent to $\ul{a}$ in terms of a rational transformation in $GL(\bQ)$. To get a more general analogue it is natural to introduce phase fields for other partitions of the ambient space.

Consider, for example, a Penrose rhomb tiling $P_{\rm pen}$ and its Delone bar-joint framework $\G_{P_{\rm pen}}$. Recall that the \emph{ribbons} of ${P_{\rm pen}}$ are the 2-way infinite paths of pairwise adjacent tiles whose common edges have the same orientation. To each ribbon we may associate the straight line through the origin which is perpendicular to the tile-joining edges of the ribbon. There are 5 such lines and we call their union the \emph{ribbon figure} of ${P_{\rm pen}}$, denoted $RF({P_{\rm pen}})$. It can be shown that in fact a ribbon with line $H$ is $H$-localised. Also for every pair of ribbons which are $H$-localised there is an $H$-localised infinitesimal flex of $\G_{P_{\rm pen}}$ which is supported on the set of joints between them. This flex is an infinitesimal translation on its support. 

More generally, in our companion paper \cite{pow-qc1} we have considered parallelogram frameworks $\G_P$ for parallelogram tilings $P$ of the plane which are determined by a regular multigrid in the sense of De Bruijn \cite{deB} and Beenker \cite{bee}. Once again, ribbons are linearly localised although, for reasons of asymmetry, their directions, which are taken to \emph{define} the ribbon figure  $RF(P)$, need \emph{not} coincide with the perpendicular line directions. 
Also, we have defined a zero mode spectrum for parallelogram frameworks $G_P$ in terms of multivariable phase fields associated with partitions of the ambient space determined by the ribbons. In the case of a Penrose tiling the associated wave vectors are vectors in a 5-dimensional reciprocal space. In what follows below we take a quite different approach, valid for general Delone bar-joint frameworks, and which is based on standard parallelepiped partitions but for variable bases. The main idea, motivated by the crystallographic case, is that when there exist linearly localised infinitesimal flexes then, in the presence of some form of  aperiodic order one may expect that there exist infinitesimal flexes that are approximately phase-periodic.

\subsection{Ribbon shears and approximately phase-periodic flexes.}\label{ss:approxphaseperiodic} 
We give some observations for parallelogram frameworks which provide a motivation for the formulation of zero mode line spectra for general Delone frameworks in the plane.

For a regular multigrid parallelogram tiling, distinct ribbons
which are $H$-localised for the same line $H$ do not cross \cite{pow-qc1}. 
Also 
$H$-localised ribbons which are of the same component grid type 
are \emph{relatively dense} in the following sense. If $M$ is  any line not parallel to $H$ then the intersection of $M$ with this subset of ribbons is a relatively dense subset of $M$. Equivalently, there exists $c>0$ such that for every pair of parallel lines $z_1+H, z_2+H$, with separation greater than $c$, the closed band they define contains a ribbon of the component grid type. 

This structure of approximately  periodical occurring $H$-localised ribbons of the same type leads to the presence of infinitesimal flexes that are approximately phase-periodic for a phase line which is reciprocal to $H$.
More precisely, consider a periodic partitioning of $\bR^2$ by $H$-localised bands $B_n, n\in \bZ$, where the bands are semiopen and have a common width $d>c$. Choose ribbons $\rho_n$ in $ B_n$, of common grid type, and corresponding \emph{shearing flexes} $u^{\rho_n}$. These infinitesimal flexes have velocity vectors equal to zero on the joints on one side of the ribbon $\rho_n$ (relative to an orientation of $M$) and a fixed common velocity ${\bf b}$ on the other side. 
Thus every difference $u^{\rho_n}-u^{\rho_m}$, for $m<n$, has support contained in the union of the bands $B_m, B_{m+1}, \dots , B_n$, and the nonzero velocity vectors are equal to {\bf b}.

 For a positive integer $N$ the infinitesimal flexes $u^{\rho_{kN-1}} - u^{\rho_{(k-1)N+1}}$ have disjoint supports, lying between the bands $B_{(k-1)N}$ and $B_{kN}$, and so for each unimodular complex number $\lambda$ we may define the infinitesimal flex
\[
u^{N,\lambda} = \sum_{k\in \bZ} \lambda^k(u^{\rho_{kN-1}} - u^{\rho_{(k-1)N+1}}).
\]
For large $N$ the velocity field vectors for $u^{N,1}$  are ``most often" equal to {\bf b}, and in this sense $u^{N,1}$ is ``mostly close" to infinitesimal translation by {\bf b}. Moreover, the associated ``phase-modulated" velocity fields $u^{N,\lambda}$ are infinitesimal flexes that are close, in the same sense, to phase-periodic velocity fields. We make these connections more precise in the subsequent sections.

\subsection{Banded phase-fields and slippage flexes} 
Let $\ul{{\bf t}}=\{{\bf t_1,t_2}\}$ be a basis and let $\B=\{B_k, k\in \bZ\}$ be the partition of $\bR^2$ by the bands 
Then  $\phi_{(1,\lambda), \ul{\bf t}}$ is equal to $\lambda^k$ on $B_k$ and we refer to this as a \emph{banded phase field}.

Let $\tau_b$ be the velocity field on $\bR^2$ which is constant and equal to the vector $b\in \bR^2$. The restriction of
$\tau_b$ to the set of joints of a Delone framework $\G$ is a translational infinitesimal flex. The \emph{modulation} of $\tau_b$ by 
$\phi_{(1,\lambda), \ul{\bf t}}$ is the pointwise product
$$
(\phi_{(1,\lambda), \ul{\bf t}}\cdot\tau_b)(x,y)
=(\phi_{(1,\lambda), \ul{\bf t}}(x,y))(\tau_b(x,y)).
$$

In general, if $\phi$ is a real or complex scalar field on the ambient space and $u$ is a real or complex velocity field on a bar-joint framework, then we similarly define the velocity field $\phi\cdot u$, which we refer to as the \emph{modulation} of $u$ by $\phi$. 

\begin{definition}\label{d:slippage}
A \emph{slippage velocity field} (resp. \emph{slippage flex}) for a Delone framework $\G$ in $\bR^2$ is a velocity field $u\in \V(\G)$ (resp. $\F(\G)$) of the form $u = \chi_S\cdot \tau_b$, where $S$ is an $H$-localised set, for some line $H$, which contains the support of $u$.
\end{definition}

In particular, as observed in Section \ref{ss:approxphaseperiodic}, a Penrose tiling framework is rich in slippage flexes supported on sets of nodes between a pair of ribbons of the same type.

The following definition  
formalises our earlier indication of two velocity fields being ``mostly close". It requires that the proportion of joints where the local velocities differ by more than $\epsilon$, in Euclidean norm, can be arbitrarily small in all squares of a given size. 

\begin{definition}\label{d:epsilonN} 
Let $u, z$ be velocity  fields   on the set $\J$ of joints of a Delone bar-joint framework  in $\bR^2$ and let $\epsilon, N$ be positive.
Then  $u$ and  $z$ are \emph{$(\epsilon, N)$-close}  if 
\[
{|\{p\in \J:\|u(p)-z(p)\|_2>\epsilon\}\cap[-N,N]^2|}
<\epsilon {|\J\cap [-N,N]^2|}.
\]
Also  $u$ and  $z$ are \emph{uniformly $(\epsilon, N)$-close} if this inequality  holds with $[-N,N]^2$ replaced by any of its translates, $(a,b)+[-N,N]^2$, for $(a,b) \in \bR^2$, and  $u$ and $z$ are \emph{mostly $\epsilon$-close} if they are uniformly $(\epsilon, N)$-close for all $N\geq N_0$, for some $N_0$. 
\end{definition}

The next somewhat technical phase-field based definition of a \emph{periodic slippage line} and the \emph{periodic slippage figure} of a Delone framework in $\bR^2$ is motivated in part by the approximately phase-periodic flexes of parallelogram frameworks discussed in Section \ref{ss:approxphaseperiodic}. We show subsequently that one can also define this figure in terms of linearly localised flexes appearing in periodic bands. Both forms of the definition give useful insights. In particular the phase-field formulation is closer to the definition of the RUM spectrum and allows modification to more general variants.

One might note the important distinction in the next definition that we are considering velocity fields that are infinitesimal flexes and that are also approximately phase-periodic, rather than phase-periodic velocity fields that are approximate infinitesimal flexes in some sense.  

\begin{definition}\label{d:periodicslippageH}
Let $H= \bR{\bf t_1}$ be a line in $\bR^2$.  Then $H$ is a \emph{periodic slippage line} for the Delone bar-joint framework $\G$ in $\bR^2$ if there exists a nonzero velocity vector $b$ in $\bR^2$ and a vector ${\bf t_2}$, with $\ul{\bf t}=\{{\bf t_1},{\bf t_2}\}$ a basis, such that for every $\epsilon>0$ there exists an infinitesimal flex $u$ of $\G$ and a positive integer $M$ with the following properties.
\medskip

(i) The modulation $\phi_{(1,\lambda), \{{\bf t_1}, M{\bf t_2}\}}\cdot u$ is an infinitesimal flex, for all $\lambda \in \bT$.
\medskip

(ii) The modulation $\phi_{(1,\lambda), \{{\bf t_1}, M{\bf t_2}\}}\cdot u$ is mostly $\epsilon$-close
to the restriction of the velocity field 
$\phi_{(1,\lambda), \{{\bf t_1}, M{\bf t_2}\}}\cdot \tau_b$, for all $\lambda \in \bT$.
\medskip

Also, the \emph{periodic slippage figure}, $PSF(\G)$, is the line figure given by the union of the periodic slippage lines of $\G$.
 \end{definition}
 
A periodic slippage line $H$ is special in the following ways. 
If $M$ is sufficiently large then there exist nonzero $H$-localised infinitesimal flexes with disjoint supports in every band
\[ 
C_k=\bR {\bf t_1}\times [kM{\bf t_2},(k+1)M{\bf t_2}).
\]
To see this, let $\lambda_0,\dots , \lambda_{p-1}$ be the distinct $p^{th}$ roots of unity and let $u$ be an infinitesimal flex with the property (i) of Definition \ref{d:periodicslippageH}. Then  the infinitesimal flex
\[
w=\frac{1}{p}\sum_{j=0}^{p-1} \phi_{(1,\lambda_j),\{{\bf t_1},M{\bf t_2}\}}\cdot u
\]
has the property $\chi_S\cdot w = w$ where $S$ is the union of the bands $C_k,$ for  $k =0 $ mod $p$. Since $\G_P$ is a Delone framework there is an upper bound to the lengths of the bars and we may take $M$ so that $pM$ is greater than this upper bound. This implies that for each integer $k$ the velocity field $\chi_{C_k}\cdot w$ is an infinitesimal flex of $\G_P$. Also, $\chi_{C_k}\cdot w$ is mostly $\epsilon$-close to $\chi_{C_k}\cdot \tau_b$.

Observe next that the banded phase field 
$\ol{\lambda_j}\phi_{(1,\lambda_j),\{{\bf t_1},M{\bf t_2}\}}$ is equal to the $M{\bf t_2}$-translate of the phase field $\phi_{(1,\lambda_j),\{{\bf t_1},M{\bf t_2}\}}$. Repeating the averaging argument above, using the infinitesimal flexes $\ol{\lambda_j}\phi_{(1,\lambda_j), \{{\bf t_1}, M{\bf t_2}\}}\cdot u$, gives a similar nonzero infinitesimal flex $w'$ with support the union of the bands $C_k$ with $k =1$ mod $p$.
More generally (considering $\ol{\lambda_j}^l$ for $l = 2,\dots ,p-1$) we obtain similar nonzero infinitesimal flexes with support the union of the bands $C_k$ for any $k$. In this way we obtain a sequence of infinitesimal flexes $u_k$, indexed by integers $k$, with support in $C_k$, such that $\sum_ku_k$ is an infinitesimal flex which is mostly $\epsilon'$-close to $\tau_b$ where $\epsilon'=p\epsilon$.

We can now state the following equivalence of definitions for a periodic slippage line.

\begin{prop}\label{p:PSFequivalence}
Let $\G$ de a Delone bar-joint framework in $\bR^2$. Then the following are equivalent.
\medskip

(i) The line $H=\bR {\bf t_1}$ given by the vector ${\bf t_1}$ is a periodic slippage line.

\medskip

(ii) For any basis $\{{\bf t_1}, \bf{t_2}\}$ and $\epsilon >0$ there exists $M>0$  such that for each band $C_k$ of the band partition for
$\{{\bf t_1}, M\bf{t_2}\}$ there is an infinitesimal flex $u_k$, with support in $C_k$, such that the sum $\sum_k u_k$ is mostly $\epsilon$-close to the restriction of $\tau_b$ for some $b\in \bC^2$. 
\end{prop}

\begin{proof}
That (i) implies (ii) has already been shown.  Assume that (ii) holds for $\epsilon$ and $M$, and  $u=\sum_k u_k$ so that $u$ and $\tau_b$
 are mostly $\epsilon$-close. Observe that $\phi_{(1,\lambda), \{{\bf t_1}, M{\bf t_2}\}}\cdot u$ is an infinitesimal flex, for each $\lambda \in \bT$ (since it is an infinite sum of infinitesimal flexes of the form $\alpha_ku_k)$. It follows that $u, M, b$
satisfy both requirements of Definition \ref{d:periodicslippageH}. 
\end{proof}

\begin{thm}\label{t:slippageequalsribbon}
Let $\G_P$ be the parallelogram framework of a regular multigrid $P$. Then  the periodic slippage figure $PSF(\G_P)$ is equal to the ribbon figure $RF(P)$. 
\end{thm}

\begin{proof}
The discussion in sections \ref{ss:G_Pandphasefields} and \ref{ss:approxphaseperiodic} show that each line $H$ in the ribbon figure $RF(P)$ satisfies the requirements of a periodic slippage line.
On the other hand if $H$ is a periodic slippage line then by the averaging argument above there exists a nonzero $H$-localised infinitesimal flex. By \cite{pow-qc1} $H$ is necessarily a line in $RF(P)$.
\end{proof}

\begin{example}\label{e:weakbracedpenrose}
Consider a Penrose tiling framework $\G_P$ which is augmented with bars and joints, for every tile, with the geometry shown in Figure \ref{f:2tilesAugmented}. The resulting framework $\G_P^+$ has essentially the same infinitesimal flex space in that the restriction map $\F(\G_P^+) \to \F(\G_P)$ is an isomorphism. Note however that $\G_P^+$ has no slippage flexes. Nevertheless this augmented framework has $H$-localised infinitesimal flexes where the nonzero velocities at the joints are mostly equal. The variation occurs for the newly added degree 2 joints that are on the boundary of the support. Thus, by taking sufficiently wide support bands we may find infinitesimal flexes $u_k$, as in Proposition \ref{p:PSFequivalence}(ii). In this way it follows that $PSF(\G_P^+)=PSF(\G_P)=RF(P)$.

\begin{center}
\begin{figure}[ht]
\centering
\includegraphics[width=7cm]{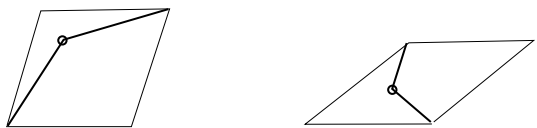}
\caption{Added joints and bars.}
\label{f:2tilesAugmented}
\end{figure}
\end{center}
\end{example}


\begin{example}
\label{e:kagomeNoSlippage} 
We note that the kagome framework $\C_{\rm kag}$ has no periodic slippage lines. 
It is known that every infinitesimal flex is a unique infinite linear combination of basic infinitesimal flexes, whose support joints lie in a line and whose velocities have two alternating directions \cite{bad-kit-pow-2}, \cite{pow-poly}. If $z$ is a nonzero $H$-localised infinitesimal flex then it follows readily from this fact that $z$ is a finite linear combination of these basic flexes where the support sets are disjoint and parallel. Since the velocity vectors alternate it follows that for sufficiently small $\epsilon$ no sum $\sum_{k\in \bZ} z_k$ of $H$-localised flexes can be mostly $\epsilon$-close to a translation infinitesimal flex.  The linear figure $PSF(\C_{\rm kag})$ is therefore equal to the empty set.
\end{example}

\begin{definition}\label{d:slippagespec} Let $\G$ be a Delone framework in $\bR^2$ and let $\ul{a}$ be a basis. 
\medskip

(1) The \emph{unreduced slippage spectrum} ${\bf L}_{\rm slip}(\G,\ul{a})$ is the line figure in the reciprocal space of $\ul{a}$
given by $PSF(\G)^{\ul{a}}= \theta_{\ul{a}}(PSF(\G))$.
\medskip

(ii) The \emph{reduced slippage spectrum} is the subset of $[0,1)^2$ given by 
\[
\Omega^{\rm log}_{\rm slip}(\G,\ul{a})= {\bf L}_{\rm slip}(\G,\ul{a})/\bZ^2 = PSF(\G)^{\ul{a}}/\bZ^2.
\]
\end{definition}

From Theorem \ref{t:slippageequalsribbon} we immediately obtain the following.

\begin{thm}\label{t:slippageSpecandRibbonFig}
Let $\G_P$ be a regular multigrid parallelogram bar-joint framework and let $\ul{a}$ be a basis for $\bR^2$. Then 
\[
\Omega^{\rm log}_{\rm slip}(\G_P,\ul{a}) = RF(P)^{\ul{a}}/\bZ^2.
\]
\end{thm}

\begin{example}\label{e:basicgridslippage}
The most elementary parallelogram tiling bar-joint framework is the grid framework $\C_{\bZ^2}$. For the standard periodicity basis $\ul{a}=\{(1,0),(0,1)\}$ the reduced slippage spectrum $\Omega^{\rm log}_{\rm slip}(\C_{\bZ^2},\ul{a})$   is the union of the two line segments $[1,0)\times \{0\}$ and $\{0\}\times [0,1)$. One can check that this set is also the logarithmic form of the RUM spectrum $\Omega(\C_{\bZ^2},\ul{a})$. See also Theorem \ref{t:slippageEqualsRUMforG_P}. 

The slippage spectrum relative to a general basis  $\ul{a}'=\{(\alpha_1,\beta_1), (\alpha_2,\beta_2)\}$ may also be computed. By Lemma \ref{l:changebasisformula}(ii) ${\bf L}_{\rm slip}(\C_{\bZ^2},\ul{a}')$ is the union of the lines
\[
Z\bR(1,0)= \bR(\alpha_1, \alpha_2), \quad  Z\bR(0,1)= \bR(\beta_1, \beta_2).
\]
Thus $\Omega^{\rm log}_{\rm slip}(\C_{\bZ^2},\ul{a})$, which is the periodic reduction of this set, is a compact set if and only if the vectors $(\alpha_1, \alpha_2), (\beta_1, \beta_2)$ do not have  directions with irrational gradients. 
\end{example}

\begin{example}\label{e:penroseslippage}
For a Penrose rhomb tiling $P$ for a regular pentagrid \cite{deB}, \cite{pow-qc1} the ribbon figure $RF(P)$ consists of 5 lines through the origin with 10-fold symmetry. In particular there is a line with irrational slope. For the standard basis $\ul{b}=\{(1,0),(0,1)\}$ the reciprocal figure $RF(P)^{\ul{b}}$ consists of 5 lines including a line of irrational slope and so the slippage spectrum  $\Omega^{\rm log}_{\rm slip}(\G_P,\ul{b})$ is a noncompact dense set. Figure \ref{f:penroseslippagespec} (b) is indicative of this. 
\begin{center}
\begin{figure}[ht]
\centering
\includegraphics[width=5cm]{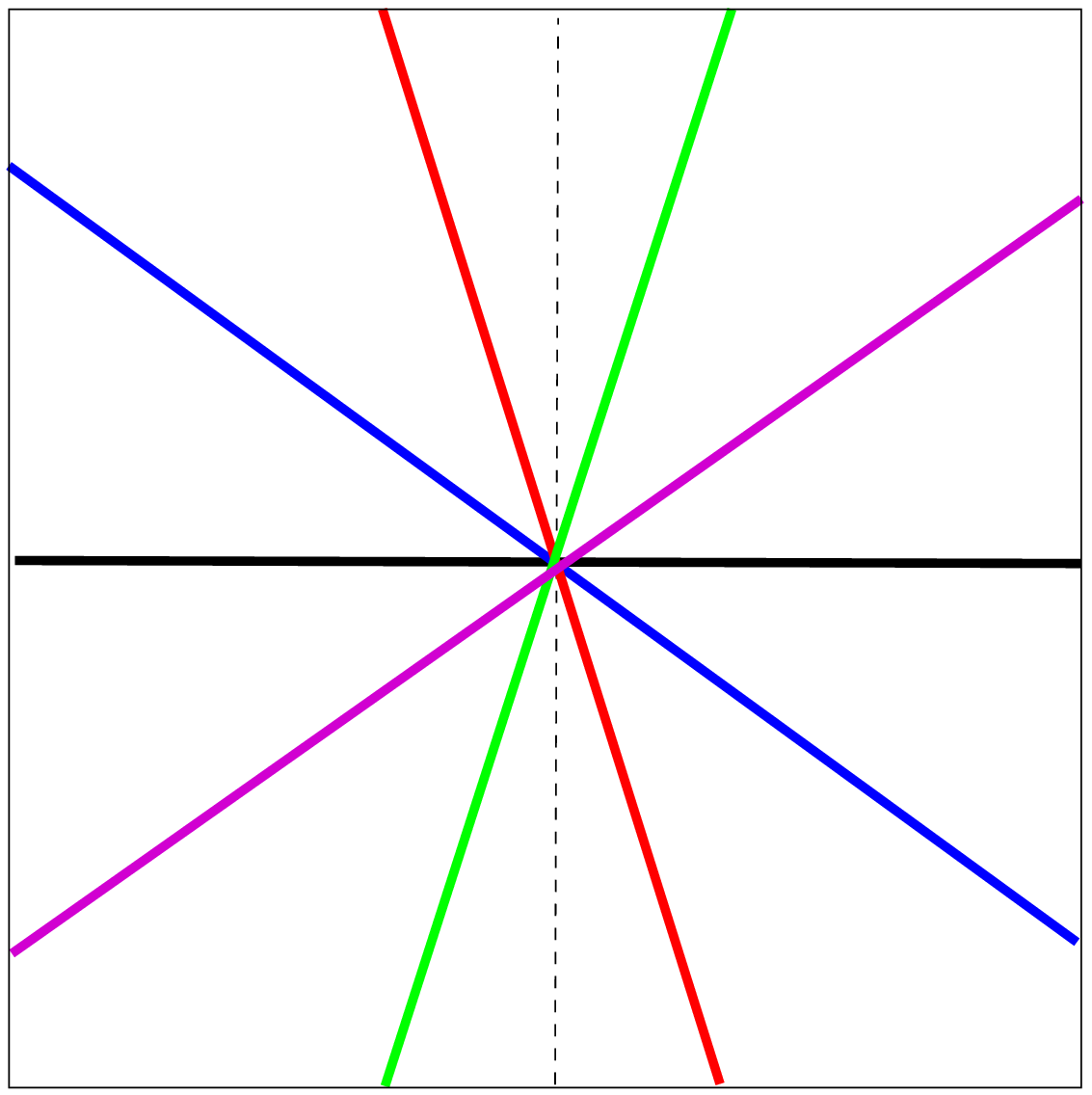}\quad \quad 
\includegraphics[width=5cm]{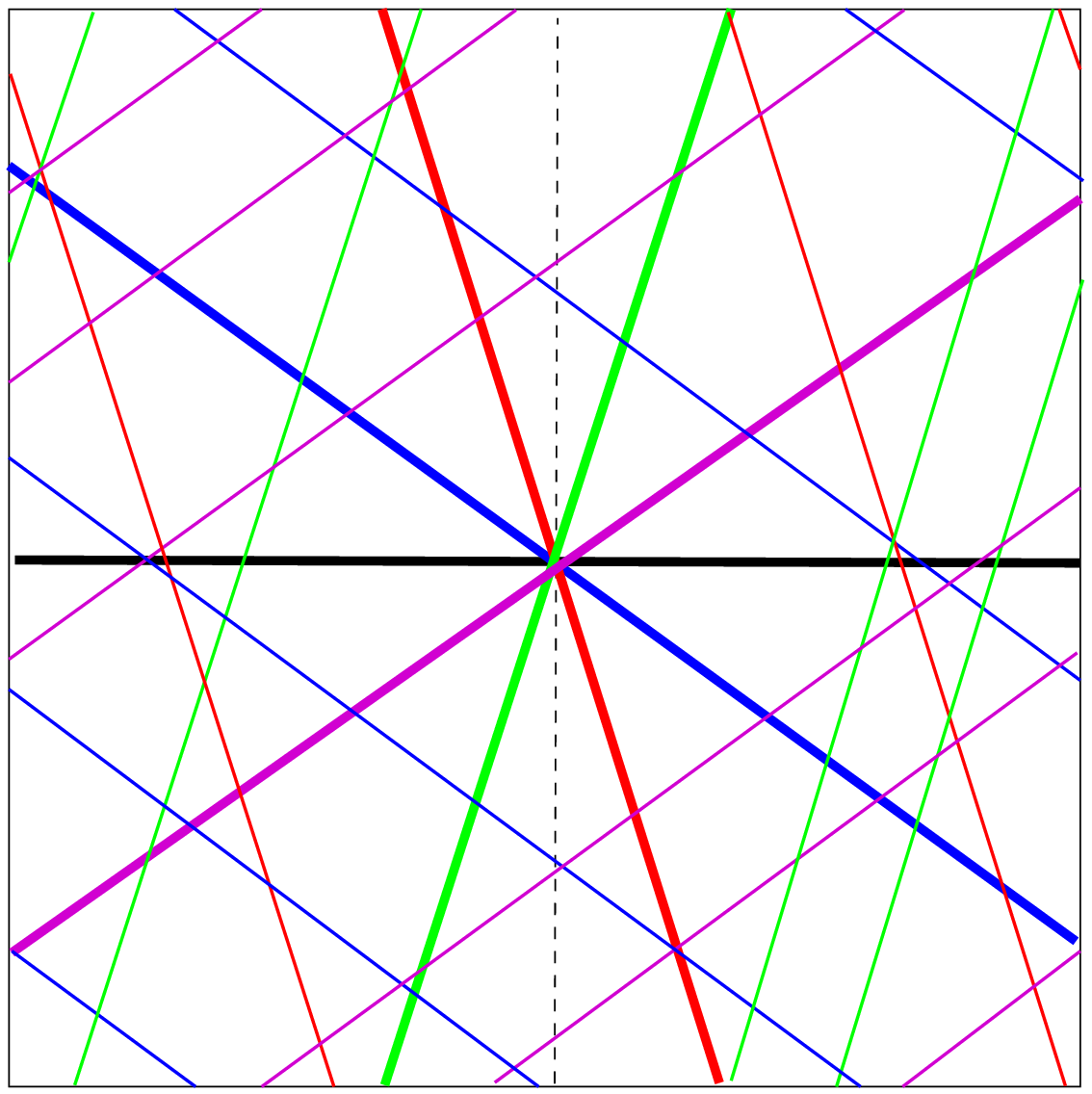}
\caption{ (a) Part of the reciprocal ribbon figure $RF(P)^{\ul{b}}$ of a Penrose tiling $P$ with respect to the standard basis $\ul{b}$. (b) Some of the dense set of lines in the reduced slippage spectrum of $\G_P$, translated to $[-1/2,1/2)^2$.}
\label{f:penroseslippagespec}
\end{figure}
\end{center}
\end{example}

Consider a rational approximation to a Penrose rhomb tiling $P$ by a periodic parallelogram tiling $P'$.  As is well known, one can construct such approximants by the projection method \cite{baa-gri}, \cite{ent-kle-pav}, \cite{gah-rhy}, \cite{ste-lub}.  
For definiteness let us say that the tilings $P, P'$, as closed sets in the plane, are $(\epsilon, N)$-close if  $P\cap [-N,N]^2$ and $P'\cap [-N,N]^2$ are $\epsilon$-close in the Hausdorff metric. It follows readily from the geometric realisation of the slippage spectrum ${\bf L}_{\rm slip}(\G_P,\ul{a})$ that we have the following spectral approximation property for this metric. For a fixed reference basis $\ul{a}$ of the ambient space and for each $\epsilon>0$ there exist $N$ and  a periodic approximant $P'$ such that ${\bf L}_{\rm slip}(\G_P,\ul{a})$
and ${\bf L}_{\rm slip}(\G_{P'},\ul{a})$ are $(\epsilon, N)$-close.
We also remark that computing the RUM spectrum of periodic approximants to Penrose frameworks gives images akin to Figure \ref{f:penroseslippagespec}(b) \cite{ste}. 

\begin{thm}\label{t:slippageEqualsRUMforG_P}
Let  $\G_P$ be a regular multigrid parallelogram framework which is periodic for a basis $\ul{a}$ and so is equal to a crystallographic  framework.
Then the unreduced zero mode spectrum ${\bf K}(\G_P,\ul{a})$ is equal to the union of ${\bf L}_{\rm slip}(\G_P,\ul{a})$ and its integral translates. In particular 
\[
\Omega^{\rm log}(\G_P,\ul{a})=\Omega^{\rm log}_{\rm slip}(\G_P,\ul{a}).
\]
\end{thm}

\begin{proof} By Corollary 2.4 of Power \cite{pow-qc1} the framework $\G_P$ has no local infinitesimal flexes and so, by Theorem \ref{t:orderN}, the RUM spectrum is proper. 
Let $L$ be a line in ${\bf K}(\C,\ul{a})$ which is parallel to  $H^{\ul{a}}$ for a line $H$ in $\bR^2$ through the origin. This is necessarily a rational line and so by Theorem \ref{t:bandconverseDim2} there exists an $H$-localised infinitesimal flex. It follows from Theorem 2.8 of \cite{pow-qc1} that $H$ is a line of the ribbon figure.

On the other hand let $H$ be a line in the ribbon figure with reciprocal line $H^{\ul{a}}$. 
Since $\ul{a}$ is a periodicity basis for $\G_P$ it follows that each $H$-localised ribbon is periodic with respect to a rational vector $b_1$ for $\ul{a}$ and that $H=\bR b_1$. By the discussion in Section \ref{ss:approxphaseperiodic} the framework $\G_P$ has an $H$-localised infinitesimal flex  of translational type, $z$ say, and this flex is periodic with respect to translation by $b_1$. Let $\ul{b}=\{b_1, b_2\}$ be a basis which is a periodicity basis for $\G_P$. Then, as in Proposition \ref{p:dualline}, we may construct the IFM
\[
u = \sum_{k_2\in \bZ} \lambda_2^{-k_2}T_{b_2}^{k_2}z, 
\]
for any $\lambda_2$ in  $\bT$. It follows that the line
$H^{\ul b}$ lies in ${\bf K}(\C,\ul{b})$. We have $H^{\ul a}= \theta_{\ul{a}}(\theta_{\ul{b}})^{-1}H^{\ul b}$ and so, by Lemma \ref{l:changebasisformula}(ii)  this line lies in ${\bf K}(\G_P, \ul{a})$. 
\end{proof}

\subsection{The limit zero mode spectrum} In spite of Example \ref{e:weakbracedpenrose} it is clear from Example \ref{e:kagomeNoSlippage} that Definition \ref{d:periodicslippageH} is quite restrictive in that the approximating phase-periodic velocity fields are based on modulations of the constant field $\tau_b$. 
To define a larger ambient space line figure we now relax the slippage line definition, replacing the velocity field $\phi_{(1,\lambda), \{{\bf t_1, Mt_2}\}}\cdot \tau_b$ with a general matricial velocity field
$\phi_{(\lambda_1,\lambda), (M_1{\bf t_1},M_2{\bf t_2})}\otimes B$. 
In this way we identify a reciprocal line figure spectrum which is analogous to the line figure $LF({\bf K}(\C,\ul{a}))$.


\begin{definition}\label{d:periodiclocalisationline} 
Let $H= \bR{\bf t_1}$ be a line in $\bR^2$. Then $H$ is a \emph{periodic localisation line} for the Delone bar-joint framework $\G$ in $\bR^2$ if 
there exists a basis $\ul{\bf t}=\{{\bf t_1},{\bf t_2}\}$ such that
for every $\epsilon>0$ there is a nonzero infinitesimal flex $u$, a phase factor $\lambda_1\in \bT$, and positive integers $M_1, M_2$, such that the following holds. 
\medskip

(i) The modulation $\phi_{(\lambda_1,\lambda), (M_1{\bf t_1},M_2{\bf t_2})}\cdot u$ is an infinitesimal flex, for all $\lambda \in \bT$.
\medskip

(ii) The modulation  $\phi_{(\lambda_1,\lambda), (M_1{\bf t_1},M_2{\bf t_2})}\cdot u$ is mostly $\epsilon$-close to the restriction of a velocity field $\phi_{(\lambda_1,\lambda), (M_1{\bf t_1},M_2{\bf t_2})}\otimes B$, for some unit cell velocity vector $B$,  for all $\lambda \in \bT$.
 \end{definition}

As with the terminology ``periodic slippage line" the adjective ``periodic" in the term ``periodic localisation line" refers to the fact that there are $H$-localised flexes appearing in periodic bands. Indeed, the first condition of Definition \ref{d:periodiclocalisationline} implies, by averaging arguments as before, that $\G$ has $H$-localised infinitesimal flexes supported in a periodic bands parallel to $H$.

We refer to the union of the periodic localisation lines as the \emph{periodic localisation line flex figure} and denote it as $PLLF(\G)$. This is a line figure in the ambient space and we have the following inclusions for Delone frameworks in $\bR^2$,
\[
PSF(\G) \subseteq PLLF(\G) \subseteq LFF(\G). 
\]
We remark that if $\G$ has some sufficiently strong form of aperiodic order and has no unbounded linearly localised flexes, then our expectation is that the linearly-localised flex figure $LFF(\G)$ will coincide with $PLLF(\G)$.

 
\begin{definition}\label{d:limitspectra} Let $\G$ be a Delone framework in $\bR^2$ and let $\ul{a}$ be a basis. 
The \emph{limit spectrum}, or \emph{limit zero mode spectrum}, of $\G$ with respect to $\ul{a}$ is the line figure
\[ 
{\bf L}_{\rm lim}(\G,\ul{a}) = PLLF(\G)^{\ul{a}}.
\]
\end{definition}

\begin{thm}\label{t:specsequalforGP}
Let $\G_P$ be a regular multigrid parallelogram bar-joint framework and let $\ul{a}$ be a basis for $\bR^2$. Then
\[
 {\bf L}_{\rm lim}(\G_P,\ul{a})=  {\bf L}_{\rm slip}(\G_P,\ul{a}) = RF(P)^{\ul{a}}.
\] 
\end{thm}

\begin{proof} By Theorem \ref{t:slippageequalsribbon} we have  
${\bf L}_{\rm slip}(\G_P,\ul{a}) = RF(P)^{\ul{a}}.$
Let $H$ be a periodic localisation line for the pair  $\G_P, \ul{a}$. Then, averaging as before, it follows that $\G_P$ has an $H$-localised infinitesimal flex. The line $H$ is necessarily a line in the ribbon figure $RF(P)$, by  \cite{pow-qc1}, and so ${\bf L}_{\rm lim}(\G_P,\ul{a})$ is contained in $RF(P)^{\ul{a}}$.
\end{proof}

\begin{thm}\label{t:elinofC}
Let $\C$ be a crystallographic bar-joint framework in $\bR^2$ with periodicity basis $\ul{a}$. Then 
\[
LF({\bf K}(\C,\ul{a})) \subseteq {\bf L}_{\rm lim}(\C,\ul{a}).
\]
\end{thm}

\begin{proof}
If ${\bf K}(\C,\ul{a}))= \bR^2$  then there exists a local infinitesimal flex $u_{\rm loc}$ by Theorem \ref{t:orderN}. Let ${\bf t}$ be a basis and $H=\bR {\bf t_1}$, with translation group $T_k, k\in \bZ^2$. Replacing ${\bf t_1}$ and ${\bf t_2}$ with $M{\bf t_1}$ and $M{\bf t_2}$ we may assume that the supports of
the velocity fields $T_ku_{\rm loc}$, for $k\in \bZ^2$, are disjoint. For any given multiphase $\omega$ the velocity field
\[
u^\omega= \sum_{k\in \bZ^2} \omega^{-k} T_ku_{\rm loc}
\]
is an $\omega$-periodic infinitesimal flex with respect to {\bf t}.
Moreover, $u^{(1,\lambda)}$ is equal to the modulation $\phi_{(1,\lambda),({\bf t_1},{\bf t_2})}\cdot u^{(1,1)}$. Thus the requirement for $H$ to be a periodic localisation line is satisfied (exactly for all $\epsilon$) and so equality of the line figures follows in this case.
 
Suppose next that ${\bf K}(\C,\ul{a})$ is proper and contains a line $L$. We show that $L$ is a periodic localisation line. Indeed, the definition of a periodic localisation line has been modelled on the crystallographic case, and we check that the conditions of the definition hold exactly in this case also.

 By the remarks following Corollary \ref{c:dualline} the line $L$ is rational with respect to $\ul{a}$. By Theorem \ref{t:bandconverseDim2}, $\C$ has an $H$-localised phase-periodic flex $u_{\rm loc}$ for a rational line $H=\bR{\bf t_1}$, with a phase-factor $\lambda_1\in \bT$, where $H^{\ul{a}}$ is parallel to $L$. Let ${\bf t_2}$ be a rational vector for $\C$ with ${\bf t}=\{{\bf t_1}, {\bf t_2}\}$ a basis. Now for 
any $\lambda_2\in \bT$, we may define
\[
u =  \sum_{k_2\in \bZ} \lambda_2^{-k_2} S_2^{k_2}u_{\rm loc}
\]
where $S_2$ is the translation isometry for $ {\bf t_2}$.
Now, taking $M_1=M_2=1$, the conditions (i) and (ii) of Definition \ref{d:periodiclocalisationline} for $H$ to be a periodic localisation line are satisfied exactly (for all $\epsilon$), where $B$ is a unit cell velocity vector matrix determined by the velocity vectors of $u$ on the joints of $\C$ in the unit cell for the partition defined by ${\bf t}$. 
\end{proof}

\begin{example}\label{e:kag_equality}
Let $\ul{a}$ be any periodicity basis for the kagome framework $\C_{\rm kag}$. It is well-known that ${\bf K}(\C_{\rm kag},\ul{a})$ consists of the union of the integral translates of 3 lines through the origin. These 3 lines are the lines $H_1^{\ul{a}}, H_2^{\ul{a}}, H_3^{\ul{a}}$ which are reciprocal to the 3 lines $H_1, H_2, H_3$ in ambient space for the 3 linear directions of the kagome tiling edges. We claim that ${\bf L}_{\rm lim}(\C_{\rm kag},\ul{a})$ is precisely the union of the 3 reciprocal lines, and so is equal to $LF({\bf K}(\C_{\rm kag},\ul{a}))$.

In view of the previous theorem it suffices to show that a line $H$ of $PLLF(\C_{\rm kag})$ is one of the 3-lines $H_1, H_2, H_3$. 
From the definition of $PLLF(\C_{\rm kag})$ and the averaging argument, there exists a nonzero $H$-localised infinitesimal flex $u$ of $\C_{\rm kag}$. 
We now use the fact that the infinitesimal flex space of $\C_{\rm kag}$ has an explicit countable {free basis} of infinitesimal flexes, each of which is $H_i$-localised for some $i$. See  \cite{bad-kit-pow-2} for example. 
Recall that a \emph{free basis}\cite{bad-kit-pow-2} of a vector subspace $\V$ of velocity fields for a countable bar-joint framework is a finite or countable subset $u_1, u_2, \dots $ with the property that each $u$ in $\V$ can be written as a unique linear combination $u= \sum_{k} \alpha_ku_k$. From this we see that $u$ must be a finite linear combination of $H_j$-localised infinitesimal flexes, for some $j$, and so $H=H_j$ as desired.
\end{example}

The argument in the previous example similarly applies to any crystallographic framework in $\bR^2$ which has a free basis of linearly localised flexes. Although not every crystallographic framework is in posession of a free basis it seems plausible that the equality ${\bf L}_{\rm lim}(\C,\ul{a})=LF({\bf K}(\C,\ul{a}))$ holds in general.

 By defining the limit spectrum as a union of lines through the origin we have not sought any finer limiting information that might be given by the phase factors of localised flexes. In particular the isolated zero modes (Weyl modes) of a crystallographic framework are not reflected in the limit spectrum.
It would be interesting to take account of some such information in the case of quasicrystal frameworks. In particular this would be relevant for  braced Penrose rhomb frameworks given in \cite{pow-qc1} that have a finite dimensional infinitesimal flex space.

\section{Further directions}\label{s:furtherdirections}
In 3 dimensions a line $L$ in the zero mode spectrum ${\bf K}(\C,\ul{a})$ of a crystallographic framework corresponds to $H$-localised infinitesimal flexes for a plane $H$ though the origin. Here the plane $H$ has a reciprocal line $H^{\ul{a}}$ which is parallel to $L$. See Section \ref{ss:lines_imply_localised}. Also, the slippage spectrum and limit spectrum of a Delone bar-joint framework in $\bR^3$ can be similarly defined, as a reciprocal line figure of an appropriate hyperplane figure in the ambient space. It would be interesting to compute such linear spectra, either analytically or computationally, for particular 3D aperiodic frameworks and more structured quasicrystals. 

In 2 dimensions we note the following further problems.
\medskip

1. It would be natural to determine linear zero mode spectra for  frameworks which have aperiodic order by virtue of being derived from  a multigrid parallelogram tiling $P$ in some well-defined local manner. As well as simple augmentations of $\G_P$ by bars and joints, such as  bracing bars \cite{pow-qc1}, or jointed bracing bars as in Example \ref{e:weakbracedpenrose}, such frameworks may be derived from $G_P$ by substitution rules or periodic Henneberg moves.
\medskip

2.  As we have remarked in Section \ref{ss:G_Pandphasefields}, for a regular multigrid parallelogram framework with $r$ component grids we have defined, in the companion article \cite{pow-qc1}, a multivariable reduced zero mode spectrum $\Omega(\G_P, \A)$ in $\bT^r$. This is based on  multivariable phase fields for a single ``patchwork" partition of the ambient space $\bR^2$ which is defined by \emph{all} the ribbons. Here $\A$ is a set of $r$ affine transformations which determine the the component grids in terms of a reference grid.  Approximation does not feature in this definition and $\Omega(\G_P, \A)$  can be viewed as a straight generalisation of the RUM spectrum of the grid framework $\C_{\bZ^2}$.
It would be natural to determine connections between the linear zero mode spectra considered above and projections of the multivariable zero mode spectrum, both for parallelogram frameworks and their derived frameworks.
\medskip

3. The slippage spectrum and limit spectrum are defined in terms of  infinitesimal flexes with strictly localised supports. On the other hand bar-joint frameworks which are random approximants or generic approximants to a quasicrystal framework are more likely to posses ``approximately localised flexes", that is, flexes with exponential decay away from a linear direction. This suggests that it would be appropriate to define less strict forms of the linear zero mode spectra given here in order to capture this, and to explore other forms of approximation in place of Definition \ref{d:epsilonN} for this purpose.

\medskip
\bibliographystyle{abbrv}
\def\lfhook#1{\setbox0=\hbox{#1}{\ooalign{\hidewidth
  \lower1.5ex\hbox{'}\hidewidth\crcr\unhbox0}}}

\end{document}

%% file: overlappingbarsLATEX.pstex_t
\begin{picture}(0,0)%
\includegraphics{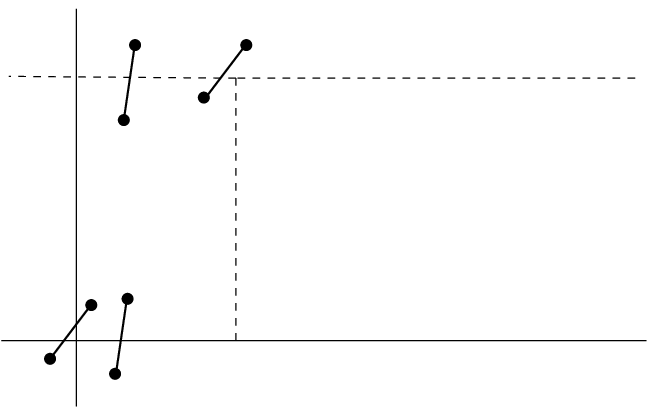}%
\end{picture}%
%
%
\setlength{\unitlength}{2368sp}%
\begingroup\makeatletter\ifx\SetFigFont\undefined%
\gdef\SetFigFont#1#2#3#4#5{%
  \reset@font\fontsize{#1}{#2pt}%
  \fontfamily{#3}\fontseries{#4}\fontshape{#5}%
  \selectfont}%
\fi\endgroup%
\begin{picture}(5184,3204)(2584,-4693)
\put(3676,-3181){\makebox(0,0)[lb]{\smash{{\SetFigFont{10}{12.0}{\rmdefault}{\mddefault}{\updefault}{\color[rgb]{0,0,0}$R_m$}%
}}}}
\put(6706,-3181){\makebox(0,0)[lb]{\smash{{\SetFigFont{10}{12.0}{\rmdefault}{\mddefault}{\updefault}{\color[rgb]{0,0,0}$\tilde{R}_m$}%
}}}}
\end{picture}%

%% file: zeromodes_arxiv_2.bbl
\begin{thebibliography}{30}





\bibitem{baa-gri} M. Baake and U. Grimm, Aperiodic Order, Vol. 1, A Mathematical Invitation, C.U.P. 2014, 545pp.

\bibitem{bad-kit-pow-1} G. Badri, D. Kitson and S. C. Power, The almost periodic rigidity of crystallographic bar-joint frameworks, Symmetry 6 (2014), 308-328. https://www.mdpi.com/2073-8994/6/2/308/xml

\bibitem{bad-kit-pow-2} G. Badri, D. Kitson and S. C. Power, Crystal flex bases and the RUM spectrum, arXiv:1807.00750, Proc. of the Edinburgh Math. Soc., First View , pp. 1 - 27, DOI: https://doi.org/10.1017/S0013091521000389

\bibitem{bee} F. P. M. Beenker,  Algebraic theory of non-periodic tilings of the plane by two simple building blocks: a square and a rhombus,  EUT report. WSK, Dept. of Mathematics and Computing Science; Vol. 82-WSK-04, Eindhoven University of Technology, 1982.





\bibitem{con-she-smi}    R. Connelly, J. D. Shen, A. D. Smith, Ball Packings with Periodic Constraints, Discrete and Computational Geometry, 52 (2014), 754-779.




\bibitem{deB} N. G. De Bruijn, Algebraic theory of Penrose's non-periodic tilings of the plane I and II, in Koninklijke
Nederlandse Academie v. Wetenschappen, series A
84 1981. pp 37-66.

\bibitem{dov-exotic} M. T. Dove, A. K. A. Pryde, V. Heine and K. D. Hammonds. Exotic distributions of rigid unit modes in the reciprocal spaces of framework aluminosilicates, J. Phys., Condens. Matter 19 (2007) doi:10.1088/0953-8984/19/27/275209.

\bibitem{dov-2019} M. T. Dove, Flexibility of network materials and the Rigid Unit Mode
model: a personal perspective. Phil. Trans. R.
Soc. A 377: 20180222.
http://dx.doi.org/10.1098/rsta.2018.0222.


\bibitem{ent-kle-pav} O. Entin-Wohlman, M. Kl\'eman and  A. Pavlovitch, Penrose tiling approximants, Journal de
Physique, 49 (1988),587-598.



\bibitem{gah-rhy} F. G\"ahler and J. Rhyner, Equivalence of the generalised grid and projection methods for the construction of quasiperiodic tilings, 1986 J. Phys. A: Math. Gen. 19 267.


\bibitem{gid-et-al}
A. P. Giddy, M. T. Dove, G. S. Pawley, V.  Heine,
 The determination of rigid unit modes as
potential soft modes for displacive phase transitions
in framework crystal structures. Acta Crystallogr.,
A49 (1993), 697 - 703.






 


\bibitem{kan-lub} C. L. Kane and T. C. Lubensky, Topological boundary modes in isostatic lattices, Nat. Phys. 10 (2014), 39-45.


\bibitem{kas-kit-mcc}  E. Kastis, D. Kitson and J. E. McCarthy,
Symbol functions for symmetric frameworks, Journal of
Mathematical Analysis and Applications 497 (2021), DOI: 10.1016/j.jmaa.2020.124895.


\bibitem{kas-pow-synthesis} E. Kastis and S. C. Power, Algebraic spectral synthesis and crystal rigidity, 
J. of Pure and Applied Algebra,  223 (2019), 4954-4965.


\bibitem{kas-pow-flexibility} E. Kastis and S. C. Power,
The first-order flexibility of a crystal framework,
J. of Math. Anal. and App.,
504 (2021), https://doi.org/10.1016/j.jmaa.2021.125404.


\bibitem{lub-et-al} T. C. Lubensky, C. L. Kane, X. Mao, A. Souslov and K. Sun, Phonons and elasticity
in critically coordinated lattices, Rep. Progress Phys. 78(7) (2015), 073901.





\bibitem{owe-pow-crystal} J. C. Owen and S. C. Power,
Infinite bar-joint frameworks, crystals
and operator theory,  New York J. Math., 17 (2011),  445-490.

\bibitem{pow-seville} S. C. Power, 
Crystal frameworks, matrix-valued functions and rigidity operators, Operator Theory: Advances and Applications.
Volume 236, 2014: Concrete Operators, Spectral Theory, Operators in Harmonic Analysis and Approximation,  Proceedings of IWOTA 2011, Seville. 

\bibitem{pow-poly} S. C. Power, Polynomials for crystal frameworks and the rigid unit mode spectrum, Royal Society Philosophical Transactions A,  Vol. 372, No. 2008, 20120030, 2014,\\  http://dx.doi.org/10.1098/rsta.2012.0030.



\bibitem{pow-qc1} 
S. C. Power,  Parallelogram frameworks and flexible quasicrystals, Mathematical Proceedings of the Royal Irish Academy 121A, no. 1 (2021): 9-31. https://doi.org/10.3318/pria.2021.121.02.



\bibitem{roc-et-al} {D. Z. Rocklin, B. Gin-ge Chen, M. Falk, V. Vitelli, and T. C. Lubensky, Mechanical Weyl Modes in Topological Maxwell Lattices, Phys. Rev. Lett. 116, 135503 (2016).}


\bibitem{ste} O. Stenull, personal communication, June 2017.

\bibitem{ste-lub} O. Stenull and T. C. Lubensky, Penrose tilings as jammed solids,
Phys. Rev. Lett. 113, 158301, 2014.

\bibitem{sch-whi} B. Schulze and W. Whiteley, Rigidity of symmetric frameworks, Handbook of Discrete and Computational Geometry, ed.  C. D. Toth, J. O'Rourke and J. E. Goodman, 3rd. ed. Chapman and Hall/CRC Press, 2017.


\bibitem{weg} F. Wegner, Rigid-unit modes in tetrahedral 
 J. Phys.: Condens. Matter 19 (2007), 406-218.


\bibitem{zho-et-al} 
D. Zhou, L. Zhang, and X. Mao,
Topological boundary floppy modes in quasicrystals,
Physical Review, X 9, 021054 (2019), DOI:10.1103/PhysRevX.9.021054

\end{thebibliography}
